\documentclass{IEEEojcsys}

\usepackage{color,array}
\usepackage{graphicx}

\jvol{}
\jnum{}
\paper{}
\pubyear{}
\receiveddate{xx}
\accepteddate{xx}
\publisheddate{xx}
\currentdate{xx}
\doiinfo{}

\usepackage{amsmath,amssymb,amsfonts}
\usepackage{textcomp}
\usepackage[normalem]{ulem}
\usepackage{tikz}
\usepackage{cite}
\usepackage{mathrsfs}
\usepackage{comment}
\usepackage{algorithm}
\usepackage{algorithmicx,algcompatible}
\usepackage[compatible]{algpseudocode}
\usepackage{stfloats}
\usepackage[colorlinks,urlcolor=blue,linkcolor=blue,citecolor=blue]{hyperref}
\newtheorem{theorem}{Theorem}
\newtheorem{definition}{Definition}

\newtheorem{lemma}{Lemma}

\newtheorem{remark}{Remark}
\newtheorem{example}{Example}
\newtheorem{assumption}{Assumption}

\DeclareMathOperator{\rank}{\rho}

\DeclareMathOperator{\row}{Row}



\begin{document}

 \sptitle{Article Category}

\title{Functional $\mathscr{H}_{\infty}$ Filtering for Descriptor Systems
with Incrementally Quadratic Nonlinearities under Disturbances%
}
\authornote{ This work is supported by the
Council of Scientific and Industrial Research, New Delhi, India, under
Grant Ref. No. 09/1023(0039)/2020-EMR-I; and by the Science and Engineering
Research Board, New Delhi, via Grant No. CRG/2023/008861.}

\author{Rishabh Sharma\affilmark{1}}

\author{Nutan Kumar Tomar\affilmark{1} (Member, IEEE)}

\affil{Department of Mathematics, Indian Institute of Technology Patna, India} 

\corresp{CORRESPONDING AUTHOR: Rishabh Sharma (e-mail: \href{rishabh$\_$2021ma18@iitp.ac.in}{rishabh$\_$2021ma18@iitp.ac.in})}
\authornote{This is an extended version of the paper accepted for
publication in the \emph{IEEE Open Journal of Control Systems}. It includes
an appendix omitted from the published version due to space constraints.
This work was supported by the Council of Scientific and Industrial
Research, New Delhi, India, under Grant No. 09/1023(0039)/2020-EMR-I, and
by the Science and Engineering Research Board, New Delhi, under Grant
No. CRG/2023/008861.}


\begin{abstract} 
This paper develops a functional $\mathscr{H}_{\infty}$ filter for nonlinear descriptor systems subject to external disturbances. Conventional $\mathscr{H}_{\infty}$ filtering approaches for descriptor systems impose restrictive regularity assumptions and employ implicit descriptor-form filters, leading to practical implementation difficulties. Moreover, existing approaches mainly target full-or reduced-order state estimation, which is computationally inefficient when only a specific functional of the state is required. To address these limitations, the filter is formulated directly in an explicit state-space framework and can be initialized with arbitrary initial values. The filter order is chosen to be less than or equal to the dimension of the functional vector to be estimated, thereby reducing computational complexity. The considered nonlinearities are characterized using incremental quadratic constraints parameterized by appropriate multiplier matrices, which encompass Lipschitz, one-sided Lipschitz, monotone, and many other nonlinearities. Sufficient criteria for the existence of the proposed filter are established through a rank condition imposed on the system matrices together with a set of linear matrix inequalities (LMIs). Under these conditions, asymptotic stability of the estimation error dynamics is guaranteed, while the influence of external disturbances on the error is bounded within a prescribed $\mathcal{L}_{2}$-performance framework. Finally, numerical simulations demonstrate and validate the effectiveness of our theoretical results.
\end{abstract}


\maketitle

\section{Introduction}\label{sec:introduction}

State estimation and filtering for nonlinear dynamical systems have been extensively studied in recent years in view of their theoretical significance and practical applications. As a generalization of classical state-space representations described by ordinary differential equations, descriptor systems are capable of representing a broader class of dynamical behaviors and capturing physical characteristics in a more realistic manner. Descriptor systems—often referred to as singular systems or differential–algebraic equation (DAE) systems—constitute a fundamental class of dynamical models characterized by the coupling of differential and algebraic constraints. Such models arise naturally in numerous application areas, including electrical networks subject to Kirchhoff’s laws, power system dynamics, mechanical systems with kinematic constraints, and chemical process models. For additional background and illustrative examples, the readers are referred to the books \cite{dai1989singular, riaza2008differential}.

With the increasing complexity of modern engineering systems, practical descriptor models often exhibit nonlinear behavior. To the best of our knowledge, the study of nonlinear descriptor systems was initiated by Luenberger in $1979$ \cite{luenberger1979non}. Since then, substantial progress has been made in exploring their solvability and control applications. For an in-depth discussion of solutions to nonlinear descriptor systems, we refer to \cite{mehrmann2006descriptor, kunkel2006differential, lamour2013differential}.
Nevertheless, most existing works on nonlinear descriptor systems focus on specific structural forms of nonlinearities. For example, nonlinearities considered in state estimation for descriptor systems include Lipschitz nonlinearities \cite{shields1997observer, lu2006full, berger2018observers, berger2020observers}, which are the most frequently recognized, as well as one-sided Lipschitz \cite{zulfiqar2016observer,zhang2012full}, monotone \cite{gupta2018unknown, berger2020observers}, and quadratic-inequality nonlinearities \cite{yang2013nonlinear, chan2021nonlinear}. Although these studies provide substantial knowledge for specific types of nonlinear descriptor systems, they remain limited and do not adequately address more general nonlinear characteristics. Consequently, it is essential to develop approaches capable of handling a broader range of nonlinear characteristics.

It is worth noting that many nonlinearities explored in the above-mentioned studies can be represented within the framework of incremental quadratic constraints ($\delta \mathrm{QC}$) through proper selections of incremental multiplier matrices ($\delta \mathrm{MMs}$) introduced by D’Alto and Corless \cite{d2013incremental}. The $\delta \mathrm{QC}$ formulation gives a unified description for a wide range of nonlinear characteristics, in the sense that different nonlinear behaviors can be interpreted as special cases of $\delta \mathrm{QC}$ by choosing appropriate $\delta \mathrm{MMs}$. In recent years, systems satisfying $\delta \mathrm{QC}$ have received considerable attention; see \cite{zhao2018observer, zhang2020adaptive, moysis2022adaptive} and the references therein. More importantly, a limited number of works have extended the $\delta \mathrm{QC}$ framework to descriptor systems \cite{moysis2020observer, liu2023state, tripathi2025reduced}. In \cite{moysis2020observer}, a state estimation method was proposed for rectangular descriptor systems subject to $\delta \mathrm{QC}$ constraints, with particular emphasis on addressing nonlinearities appearing in the output equation. The work in \cite{liu2023state} investigates descriptor systems incorporating $\delta \mathrm{QC}$-type nonlinearities in both the output and state dynamics; however, the analysis is restricted to square system configurations, and the unknown disturbance is assumed to be decoupled from the nonlinear components. In contrast, the model considered in \cite{tripathi2025reduced} addresses rectangular descriptor systems subject to $\delta \mathrm{QC}$ nonlinearities in both the output and state equations, along with an unknown parameter in nonlinear form.

The aforementioned discussion primarily deals with observers to estimate the complete state vector of a control system. However, in many practical scenarios, estimating the entire state vector is neither necessary nor desirable. Often, only a subset or a linear function of state variables is required, as in fault detection, feedback control, and disturbance estimation. Observers that target a linear function of state variables are known as functional observers. Such observers are typically of much lower order, eliminate redundant state reconstruction, and can be developed under less restrictive conditions than full-state estimators \cite{trinh2011functional}. A comprehensive overview of functional observers for linear systems described by DAEs is reported in \cite{jaiswal2021existence, tunga2023functional, jaiswal2024existence} and the references therein.

Motivated by these considerations, addressing functional state estimation in the presence of disturbances becomes an important problem. In this context, the $\mathscr{H}_\infty$ filtering problem has emerged as a powerful and widely adopted approach for addressing state estimation problems. In $\mathscr{H}_{\infty}$ filtering, the goal is to design a stable filter that ensures the induced $\mathcal{L}_{2}$-performance from disturbance signals to the estimation error does not exceed a prescribed attenuation level. In contrast to classical methods of filtering, such as the Kalman filter, the $\mathscr{H}_\infty$ framework does not rely on any probabilistic assumptions regarding the disturbance or noise signals, which enables this framework to be used for performing a worst-case performance analysis \cite{simon2006optimal}. According to existing literature, Xu \emph{et al.} \cite{xu2003filtering} were the first to formulate the functional $\mathscr{H}_\infty$ filtering problem within the class of linear descriptor systems satisfying the regularity condition. Since its introduction, the topic has attracted sustained research interest. The design of $\mathscr{H}_\infty$ filter for linear DAE systems has been
studied in \cite{xu2007reduced,darouach2009unbiased, sahereh2017h, tunga2023h}, and for nonlinear DAE systems in \cite{darouach2011,abbaszadeh2012generalized,zhang2015improved,sharma2026functional}. Despite their practical relevance, functional observers and filters for nonlinear descriptor systems have received very little attention. For example, Zhang \emph{et al.} \cite{zhang2015improved} considered Lipschitz nonlinearities and proposed a functional $\mathscr{H}_\infty$ filter in descriptor form, which is less advantageous in practice owing to its implicit structure and the requirement for consistent initial conditions.  Recently, Sharma \emph{et al.} \cite{sharma2026functional} introduced a new concept of functional $\mathscr{H}_{\infty}$ filters for nonlinear descriptor systems characterized by generalized monotone nonlinearities, presenting the filter in state-space form, which offers a more practical and accessible approach.
   However, so far, no results are available on functional $\mathscr{H}_\infty$ 
 filtering problem for nonlinear descriptor systems that satisfy the $\delta \mathrm{QC}$ condition.
 In light of the research gaps identified in the preceding discussion, this paper is focused on functional $\mathscr{H}_{\infty}$ filter design for nonlinear descriptor systems satisfying $\delta \mathrm{QC}$. The primary contributions of this work are provided below.
 \begin{enumerate}
    \item Previous studies have mainly focused on full- or reduced-order filters for regular and square descriptor systems, whereas this paper develops a functional $\mathscr{H}_{\infty}$ filter for a class of nonlinear descriptor systems subject to disturbances. Moreover, the proposed approach explicitly accommodates general (rectangular) descriptor structures, thereby extending the applicability beyond existing approaches \cite{xu2003filtering, xu2007reduced, sahereh2017h, zhang2015improved, liu2023state}.
    \item Unlike the prior implicit descriptor form of $\mathscr{H}_{\infty}$ filters \cite{abbaszadeh2012generalized, liu2023state, zhang2015improved}, the proposed filter adopts an explicit state-space representation that simplifies implementation and allows arbitrary initial conditions.
    \item A key novelty of our approach is that the filter order is chosen to be less than or equal to the dimension of the functional vector to be estimated, thereby ensuring that only the required state components are reconstructed. 
    \item A novel set of sufficient conditions is proposed, characterized by a rank constraint on the system coefficient matrices together with an LMI framework. These conditions ensure asymptotic stability of the estimation error dynamics and bound the influence of external disturbances on the estimation error within a prescribed $\mathcal{L}_2$ performance level. Furthermore, the filter design procedure is described algorithmically and can be readily implemented using existing LMI solver packages, e.g. the LMI toolbox in MATLAB. 
\end{enumerate}

\textbf{Notations:}
We use $0$ and $\mathbb{I}$ to indicate zero and identity matrices, respectively. All unspecified blocks in partitioned matrices are zero matrices of compatible sizes. For convenience, we use the $\mathbb{I}_n$ for the $n \times n$ identity matrix and $0_{m \times n}$ for the $m \times n$ zero matrix.  The set of complex numbers is represented by $\mathfrak{C}$. For any matrix $\mathbb{A}$, the symbols $\mathbb{A}^\top$, $\mathbb{A}^{+}$, and $\rho (\mathbb{A})$ represent its transpose, Moore-Penrose (MP) inverse, and rank, respectively. For any (square) matrix $\mathbb{A}$, we denote $\mathbb{A}>0(\mathbb{A}<0)$ if, and only if, $\mathbb{A}$ is a symmetric positive (negative) definite matrix. $\mathbb{C}$ denotes the set of complex numbers and
$\overline{\mathbb{C}}^{+} := \{\lambda \in \mathbb{C} \mid \operatorname{Re}(\lambda) \geq 0\}$. We write $\mathfrak{C}^k(X \to Y)$ for the set of $k$-times continuously differentiable functions $f: X \to Y$. The space of square-integrable functions $f:[0,t_f]\to\mathbb{R}^n$, with $t_f>0$, 
is represented by $\mathcal{L}_2([0,t_f],\mathbb{R}^n)
:= \{ f : \int_0^{t_f} \|f(t)\|^2 dt < \infty \}$,
where $\|f\|_2^2 = \int_0^{t_f} \|f(t)\|^2 dt$. The notation $x(t) \to 0$ as $t \to \infty$ signifies that $\lim_{t \to \infty} \|x(t)\| = 0$ and that $\|\cdot\|$ is the Euclidean norm. Finally, $\text{blk-diag}(A_1,\ldots,A_k)$ stands for the block diagonal matrix formed with diagonal blocks $A_1,\ldots,A_k$.

 \textbf{Organization:} The paper proceeds as follows. Section~\ref{sec:prob_form} introduces the problem setup and necessary preliminaries. In Section~\ref{main}, a constructive framework for functional filter design is developed using matrix rank conditions and LMI feasibility analysis. Here, the order of the proposed filter can be selected to be less than or equal to the size of the functional vector. In Section~\ref{sec:num_example}, we provide two numerical examples to substantiate our theoretical findings. Section~\ref{sec:conclusion} summarizes the paper and outlines potential avenues for future research.


\section{ Problem formulation and preliminaries}\label{sec:prob_form}
In this paper, we focus on nonlinear descriptor systems satisfying the $\delta \mathrm{QC}$ condition, described by:
\begin{subequations} \label{sys}
     \begin{eqnarray}
         \mathcal{E}\dot{x}(t) &= & \mathcal{A}x(t)+\mathcal{B}u(t)+\mathcal{D}v(t)+\mathcal{\bar{D}}_{\xi}\xi(t) \label{sysa}\\ 
        && +\mathcal{F}g(H\mathcal{K}x,t), \notag \\
	  \Tilde{y}(t) &= & \mathcal{C}x(t)+\mathcal{G}v(t),    \label{sysb}\\
	   z(t) &= & \mathcal{K}x(t),    \label{sysc}
    \end{eqnarray}
\end{subequations}
where $\mathcal{E}, ~\mathcal{A} \in {\mathbb{R}}^{m \times n}$, $\mathcal{B} \in {\mathbb{R}}^{m \times k}$, $\mathcal{C} \in {\mathbb{R}}^{r \times n}$, $\mathcal{D} \in {\mathbb{R}}^{m \times q}$, $\mathcal{\bar{D}}_{\xi} \in {\mathbb{R}}^{m \times s}$, $\mathcal{F} \in {\mathbb{R}}^{m \times l}$,  $\mathcal{G} \in {\mathbb{R}}^{r \times q}$, $H \in {\mathbb{R}}^{l \times p}$, and $\mathcal{K} \in {\mathbb{R}}^{p \times n}$ are constant matrices. We refer to $x(t) \in \mathbb{R}^{n}$ as the semistate vector, $u(t) \in \mathbb{R}^{k}$ as the (known) input vector, $\Tilde{y}(t) \in \mathbb{R}^{r}$ as the (measured) output vector, $v(t) \in \mathbb{R}^{q}$ as the (unknown) input vector, $\xi(t) \in \mathbb{R}^{s}$ as the disturbance vector of bounded energy, and $z(t)\in \mathbb{R}^{p}$ as the (unmeasured) output functional vector. Let $\mathscr{X} \subseteq \mathbb{R}^p$ be an open set and define $\mathscr{X}_1 := H\mathscr{X} \subseteq \mathbb{R}^l$. The nonlinear mapping $g: \mathscr{X}_1 \times \mathbb{R} \to \mathbb{R}^l$ is assumed to be known and continuously differentiable, \emph{i.e.}, $g \in \mathfrak{C}^1(\mathscr{X}_1 \times \mathbb{R}  \to \mathbb{R}^l)$, and to satisfy an incremental quadratic constraint with respect to its first argument (see Assumption~\ref{assum:1}). Throughout this work, we assume that the system designer has 
defined the system variables such that system \eqref{sys} admits 
at least one solution (cf. Definition~\ref{d1}).  Furthermore, we assume that for 
admissible initial conditions, inputs, and disturbance functions, 
the functional $z(t)$ is unique even if the uniqueness of $x(t)$ is not guaranteed by \eqref{sysa}. As shown in the Appendix, the uniqueness of $z(t)$ is not
merely an assumption, but a necessary condition for the existence of any
functional filter. We define the solutions of system~\eqref{sys} as follows:

\begin{definition} [\cite{berger2018observers}]\label{d1}
Let $\mathtt{I} \subseteq \mathbb{R}$ be an open interval. A tuple 
$(x, u, v, \xi, \tilde{y}, z) \in \mathfrak{C}(\mathtt{I};\mathbb{R}^{n+k+q+s+r+p})$ 
is called a solution of~\eqref{sys} if $x \in \mathfrak{C}^{1}(\mathtt{I};\mathbb{R}^{n})$ 
and \eqref{sys} holds for every $t \in \mathtt{I}$.
\end{definition}
In this work, we develop a design methodology for a filter described by the following structure:
\begin{subequations} \label{est}
\begin{eqnarray}
\dot{w}(t) & = & N w(t)+ H_1u(t) +L_1\Tilde{y}(t) + F_1g(H\hat{z},t), ~~ \label{esta} \\
\hat{z}(t) & =& Rw(t)+M_1u(t)+M_2\Tilde{y}(t), \label{estb}
\end{eqnarray}
\end{subequations}
where $w(t) \in \mathbb{R}^{l_{1}}$ signifies the state vector of the filter, while $\hat{z}(t) \in \mathbb{R}^{p}$ denotes the estimate of the functional state $z(t)$. A filter of the above form \eqref{est} is known as an ODE filter, and the non-negative integer $l_{1}$ ($\leq p$) is referred to as its order. If $l_{1}<p$, the filter is said to be of reduced-order. Furthermore, if 
$l_{1}=0$, the filter is described solely by \eqref{estb} and is therefore referred to as a static filter (see Remark \ref{rem13}).

We now define a functional $\mathscr{H}_\infty$ filter for 
system~\eqref{sys}.
\begin{definition}\label{def2}
System~\eqref{est} is called a functional $\mathscr{H}_{\infty}$ filter 
for~\eqref{sys} if for every solution $(x, u, v, \xi, \tilde{y}, z)$ 
of~\eqref{sys}, there exist $w \in \mathfrak{C}^{1}(\mathtt{I};\mathbb{R}^{l_{1}})$ 
and $\hat{z} \in \mathfrak{C}(\mathtt{I};\mathbb{R}^{p})$ such that $(w, u, \Tilde{y}, \hat{z})$ satisfy \eqref{est} for every $t \in \mathtt{I}$, and for all such $w,~ \hat{z}$ the following properties hold:
\begin{enumerate}
\item If $\xi \equiv 0$, then $z(t) - \hat{z}(t) \rightarrow 0$ 
as $t \rightarrow \infty$.
\item If $\xi \not\equiv 0$, then
\begin{eqnarray*}
\|z - \hat{z}\|^2_{2} < \gamma^2\|\xi\|^2_{2} + \beta,
\end{eqnarray*}
for some constants $\gamma > 0$ and $\beta \geq 0$.
\end{enumerate}
\end{definition}

Mathematically, this paper aims to determine the parameter matrices  $N$, $H_1$, $L_1$, $F_1$, $R$, $M_1$, and $M_2$
such that system \eqref{est} serves as a functional
$\mathscr{H}_{\infty}$ filter for system \eqref{sys}, cf. Definition \ref{def2}. The filter structure \eqref{est} is favored for its straightforward implementation and its flexibility in choosing arbitrary initial conditions, in contrast to DAE-based filters \cite{zhang2015improved}, whose implicit formulation necessitates consistent initial conditions.

We now restate the following matrix theory results, which are essential for the subsequent discussion.
\begin{definition}\label{def:detect} \cite{chen1984linear}
The matrix pair $(A, C)$ is called detectable if and only if
\begin{equation*}
\rank\begin{bmatrix}\lambda I - A \\ C\end{bmatrix} = n, 
\qquad \text{for all } \lambda \in \bar{\mathbb{C}}^+.
\end{equation*}
Equivalently, a matrix $Z$ of compatible dimension exists such 
that the matrix $A - ZC$ is Hurwitz (stable).
\end{definition}

\begin{lemma} \label{lem2} \cite{matsaglia1974equalities}
Let $\mathrm{U}, \mathrm{W}$, and $\mathrm{V}$ be matrices of compatible sizes. If $\mathrm{U}$ has full row rank and/or $\mathrm{V}$ has full column rank, then
\begin{eqnarray*}
\rho \begin{bmatrix}
         \mathrm{U} & \mathrm{W} \\
         0 & \mathrm{V}
\end{bmatrix} = \rho (\mathrm{U}) + \rho (\mathrm{V}).
\end{eqnarray*}
\end{lemma}
\begin{lemma}\label{lem4} \cite{piziak2007matrix}
  System $\mathrm{U} \mathrm{W}=\mathrm{V}$ has a solution for $\mathrm{U}$ if and only if $\mathrm{V}=\mathrm{V} \mathrm{W}^{+} \mathrm{W}$, or equivalently, $\rho \begin{bmatrix}
  \mathrm{W} \\ \mathrm{V}
  \end{bmatrix} = \rho (\mathrm{W})$. Furthermore, 
$$
\mathrm{U}=\mathrm{V} \mathrm{W}^{+}-\mathrm{Z} (\mathbb{I}-\mathrm{W} \mathrm{W}^{+} ),
$$
where $\mathrm{Z}$ is a free matrix parameter with a compatible dimension.
\end{lemma}
\begin{lemma}\label{lem:fcr} \cite{matsaglia1974equalities}
Let $\mathrm{M}$ and $\mathrm{N}$ be any two matrices of compatible
dimensions. Then $\rho(\mathrm{M}\mathrm{N})=\rho(\mathrm{N})$ if
and only if the matrix
$\begin{bmatrix}\mathrm{M}\\ \mathbb{I}-\mathrm{N}\mathrm{N}^{+}
\end{bmatrix}$ has full column rank (FCR).
\end{lemma}
\begin{lemma}\label{lem:blockrank}
Let $\mathrm{A}$, $\mathrm{B}$, $\mathrm{C}$, and $\mathrm{D}$ be
any matrices of compatible dimensions such that
$\rho\begin{bmatrix}\mathrm{A} & \mathrm{B}\\ \mathrm{C} &
\mathrm{D}\end{bmatrix}=\rho\begin{bmatrix}\mathrm{A} & \mathrm{B}
\end{bmatrix}$. Then
$\rho\begin{bmatrix}\mathrm{A}\\ \mathrm{C}\end{bmatrix}
=\rho(\mathrm{A})$ and
$\rho\begin{bmatrix}\mathrm{B}\\ \mathrm{D}\end{bmatrix}
=\rho(\mathrm{B})$.
\end{lemma}

We conclude this section by introducing the following two assumptions about system (1).

\begin{assumption}\label{assum:1}
    The nonlinear term $g(\zeta, t)$ of system \eqref{sys} satisfies the  $\delta \mathrm{QC}$ condition:
\begin{eqnarray}\label{ineq}
\begin{bmatrix}
    \zeta_1- \zeta_2 \\
g( \zeta_1, t)-g( \zeta_{2}, t)
\end{bmatrix}^\top \mathcal{W}\begin{bmatrix}
 \zeta_1- \zeta_2 \\
g( \zeta_1, t)-g( \zeta_{2}, t)
\end{bmatrix} \geq 0 .
\end{eqnarray}
The known $\delta\mathrm{MM}$ $\mathcal{W}$ describes the incremental characteristics of $g(\zeta, t)$ and is partitioned as follows
\begin{eqnarray}\label{imm}
     \mathcal{W}= \begin{bmatrix}
\mathcal{W}_{11} & \mathcal{W}_{12} \\
\mathcal{W}_{12}^\top & \mathcal{W}_{22}
 \end{bmatrix},
\end{eqnarray}
where $\mathcal{W}_{11}$ and $\mathcal{W}_{22}$ are symmetric matrices of suitable dimensions. Table~\ref{table:IMM} provides representative $\delta\mathrm{MM}$ structures for several commonly encountered classes of nonlinearities, highlighting that these structures arise as special cases of the general $\delta\mathrm{QC}$ condition \eqref{ineq}.

\begin{table}[ht]
\caption{Nonlinearity classes and corresponding $\delta \mathrm{MM}$.}
\label{table:IMM}
\tablefont
\begin{tabular*}{20.25pc}{@{}p{82pt}p{90pt}<{\raggedright}p{52pt}<{\raggedright}@{}}
\hline\\[-1.5pc]\hline
\centerline{$\delta \mathrm{MM}$} &
\centerline{Condition} &
\centerline{Nonlinearity} \\[3pt]
$\scriptscriptstyle
\begin{pmatrix}
\lambda_{g}^{2}I & 0 \\[1pt]
0 & -I
\end{pmatrix}$ &
\centerline{$\|\Delta g\|\leq\lambda_{g}\|\Delta \zeta\|$} &
\centerline{Lipschitz} \\[22pt]
$\scriptscriptstyle
\begin{pmatrix}
\rho I & -\frac{1}{2}I \\[1pt]
-\frac{1}{2}I & 0
\end{pmatrix}$ &
\centerline{$(\Delta \zeta)^{\top}\!\Delta g\leq\rho\|\Delta \zeta\|^{2}$} &
\centerline{OSL} \\[22pt]
$\scriptscriptstyle
\begin{pmatrix}
\beta I & \frac{1}{2}\gamma I \\[1pt]
\frac{1}{2}\gamma I & -I
\end{pmatrix}$ &
\centerline{$\|\Delta g\|^{2}\leq\beta\|\Delta \zeta\|^{2}$}
\centerline{${}+\gamma(\Delta \zeta)^{\top}\!\Delta g$} &
\centerline{QIB} \\[22pt]
$\scriptscriptstyle
\begin{pmatrix}
0 & I \\[1pt]
I & 0
\end{pmatrix}$ &
\centerline{$(\Delta \zeta)^{\top}\!\Delta g\geq 0$} &
\centerline{Monotone} \\[22pt]
$\scriptscriptstyle
\begin{pmatrix}
-\mu I & I \\[1pt]
I & 0
\end{pmatrix}$ &
\centerline{$(\Delta \zeta)^{\top}\!\Delta g\geq\frac{\mu}{2}\|\Delta \zeta\|^{2}$} &
\centerline{Gen.\ monotone} \\[22pt]
\hline\\[-1.5pc]\hline
\end{tabular*}
\begin{minipage}{20.25pc}
\vspace{3pt}
{\scriptsize
Here $\Delta \zeta = \zeta_1 - \zeta_2$,\ 
$\Delta g = g(\zeta_1,t)-g(\zeta_2,t)$,\
$\lambda_g>0$,\ $\rho,\mu\in\mathbb{R}$,\ 
and $\beta,\gamma\in\mathbb{R}$.
}
\end{minipage}
\end{table}
\end{assumption}
\begin{remark}\label{rem11}
    It is worth noting that if a matrix $\mathcal{W}$ satisfies \eqref{ineq}, then any positively scaled matrix $\kappa\mathcal{W}$, with $\kappa>0$, also satisfies \eqref{ineq}. A similar discussion can be found in \cite{chakrabarty2017state}.
\end{remark}

\begin{assumption}\label{assum:2} The coefficient matrices of system \eqref{sys} satisfy 
\begin{eqnarray}\label{rank1}
\rank(\Psi) = \rank (\Omega),
\end{eqnarray}
where $\Psi =\begin{bmatrix}
\mathcal{E}   & \mathcal{A} & 0 & \mathcal{D} & 0  & \mathcal{\bar{D}_{\xi}} & \mathcal{F}\\
0  & \mathcal{E} & \mathcal{A}  & 0 & \mathcal{D}&0 & 0\\
0  & \mathcal{C} & 0  & \mathcal{G} & 0 &0 & 0\\
0  & 0 & \mathcal{C} & 0  & \mathcal{G} &0 & 0\\
0  & 0  & -\mathcal{K} & 0 & 0 &0 & 0  \\
0 & \mathcal{K}  & 0 & 0 &  0 &0 & 0
\end{bmatrix}$ and $\Omega = \begin{bmatrix}
\mathcal{E}   & \mathcal{A} & 0 & \mathcal{D} & 0  & \mathcal{\bar{D}_{\xi}} & \mathcal{F}\\
0  & \mathcal{E} & \mathcal{A}  & 0 & \mathcal{D}&0 & 0\\
0  & \mathcal{C} & 0  & \mathcal{G} & 0 &0 & 0\\
0  & 0 & \mathcal{C} & 0  & \mathcal{G} &0 & 0\\
0  & 0  & -\mathcal{K} & 0 & 0 &0 & 0  
\end{bmatrix}$.
\end{assumption}

Before providing a remark on the above Assumption~\ref{assum:2}, we summarize in Table~\ref{tab:rank} below, the rank conditions assumed in previous works on full order state estimation for nonlinear descriptor systems.
\begin{table}[h!]
\caption{Rank conditions in the existing literature.}
\label{tab:rank}
\tablefont
\begin{tabular*}{20.25pc}{@{}p{52pt}p{155pt}<{\raggedright}@{}}
\hline\\[-1.5pc]\hline
\centerline{\textbf{Reference}} &
\centerline{\textbf{Rank Condition}} \\[3pt]
\centerline{Liu \emph{et al.}~\cite{liu2023state}} &
\centerline{$\rank\begin{bmatrix}
\mathcal{E}^\top&\mathcal{C}^\top
\end{bmatrix}^\top=n$} \\[20pt]
\centerline{Darouach \emph{et al.}~\cite{darouach2011}} &
\centerline{$\rank\begin{bmatrix}
\mathcal{E}^\top&(\Phi\mathcal{A})^\top&\mathcal{C}^\top
\end{bmatrix}^\top=n$}  \\[15pt]
\centerline{Gupta \emph{et al.}~\cite{gupta2018unknown}} &
\scalebox{0.9}{$\rank\begin{bmatrix}
\mathcal{E} & \mathcal{F} & \mathcal{A} & \mathcal{D} & 0 \\
0 & 0 & \mathcal{E} & 0 & \mathcal{D} \\
0 & 0 & \mathcal{C} & \mathcal{G} & 0 \\
0 & 0 & 0 & 0 & \mathcal{G}
\end{bmatrix}
=\rank\begin{bmatrix}
\mathcal{E} & \mathcal{F} & \mathcal{D} \\
0 & 0 & \mathcal{G}
\end{bmatrix}+n+q$} \\[40pt]
\hline\\[-1.5pc]\hline
\end{tabular*}
\begin{minipage}{20.25pc}
\vspace{3pt}
{\scriptsize
Here $\Phi$ is full row rank satisfying $\Phi\begin{bmatrix}
\mathcal{E}&\mathcal{F}
\end{bmatrix}=0$}
\end{minipage}
\end{table}

\begin{remark}
Assumption~\ref{assum:2} plays a fundamental role in the proposed filter design. In particular, it guarantees the solvability of the filter design matrix equations (cf.~\eqref{cond}), thereby ensuring that the estimation error dynamics (cf.~\eqref{et}) is well defined.
Notably, when $\mathcal{K}=\mathbb{I}_{n}$,
Assumption~\ref{assum:2} reduces to
\begin{equation}\tag{A2.1}\label{eq:assum2red}
\scalebox{0.95}{$\rank
\begin{bmatrix}
\mathcal{E} & \mathcal{F} & \mathcal{A} & \mathcal{D}
& 0 & \bar{\mathcal{D}}_{\xi} \\
0 & 0 & \mathcal{E} & 0 & \mathcal{D} & 0\\
0 & 0 & \mathcal{C} & \mathcal{G} & 0 & 0\\
0 & 0 & 0 & 0 & \mathcal{G} & 0
\end{bmatrix}
=
\rank
\begin{bmatrix}
\mathcal{E} & \mathcal{F} & \mathcal{D} & 0
& \bar{\mathcal{D}}_{\xi} \\
0 & 0 & 0 & \mathcal{D} & 0\\
0 & 0 & \mathcal{G} & 0 & 0\\
0 & 0 & 0 & \mathcal{G} & 0
\end{bmatrix}
+\,n$}
\end{equation}
Setting $\bar{\mathcal{D}}_{\xi}=0$ in \eqref{eq:assum2red}, 
it is easy to verify that the rank condition of 
\cite{gupta2018unknown} is more restrictive than the 
resulting condition. Further setting $\mathcal{D}=0$ and 
$\mathcal{G}=0$ in \eqref{eq:assum2red} yields
\begin{equation}\tag{A2.2}\label{eq:darouach}
\rank
\begin{bmatrix}
\mathcal{E}^{\top} & (\Phi\mathcal{A})^{\top} & 
\mathcal{C}^{\top}
\end{bmatrix}^{\top}=n,
\end{equation}
where $\Phi$ is a full-row-rank matrix satisfying 
$\Phi\begin{bmatrix}
    \mathcal{E}&\mathcal{F}
\end{bmatrix}=0$, which is precisely 
the rank condition of \cite{darouach2011}. Moreover, the rank condition of 
\cite{liu2023state} in Table~\ref{tab:rank} is strictly stronger than \eqref{eq:darouach}.
Thus, as summarized in Table~\ref{tab:rank}, the 
rank conditions employed in \cite{liu2023state}, 
\cite{darouach2011}, and \cite{gupta2018unknown} are 
either special cases of, or strictly stronger than, the 
proposed rank condition (cf. Assumption~\ref{assum:2}). It is worth noting that the condition of \cite{darouach2011} is equivalent 
to impulse observability of the linear descriptor system 
(\emph{i.e.}, $g\equiv 0$ in \eqref{sys}), which constitutes 
a standard assumption in the observer design literature 
for linear descriptor systems; see, e.g., 
\cite{darouach1995design,
berger2017observers} and the references therein.
\end{remark}


\section{Functional $\mathscr{H}_{\infty}$ filter design}\label{main}

In order to design a functional $\mathscr{H}_{\infty}$ filter, we first transform the given system \eqref{sys} into a new coordinate system.
If $\rank\begin{bmatrix}
\mathcal{E}  & \mathcal{\bar{D}}_{\xi} & \mathcal{F} \end{bmatrix} = n_0 \leq n$, then by using its singular value decomposition (SVD), there exist a non-singular matrix $\mathcal{P}$ such that $\mathcal{P}\begin{bmatrix}
\mathcal{E}  & \mathcal{\bar{D}}_{\xi} & \mathcal{F} \end{bmatrix}= \begin{bmatrix}
\mathcal{E}_{1} & \mathcal{D}_{\xi} &\mathcal{F}_{1} \\ 0 & 0 & 0
\end{bmatrix}$, where $\begin{bmatrix}
\mathcal{E}_{1}  & \mathcal{D}_{\xi} & \mathcal{F}_{1} \end{bmatrix}$ is the full-row rank matrix, i.e., $\rank\begin{bmatrix}
\mathcal{E}_{1}  & \mathcal{D}_{\xi} & \mathcal{F}_{1} \end{bmatrix}= n_0$. Take $\mathcal{P} \mathcal{A}=\begin{bmatrix}
 \mathcal{A}_{1}\\ \mathcal{A}_{2}\end{bmatrix}$, $\mathcal{P} \mathcal{B}=\begin{bmatrix}
 \mathcal{B}_{1}\\\mathcal{B}_{2} \end{bmatrix}$, and $\mathcal{P} \mathcal{D}=\begin{bmatrix}
 \mathcal{D}_{1}\\ \mathcal{D}_{2} \end{bmatrix}$, then \eqref{sys} can be rewritten as
\begin{eqnarray*}
\mathcal{E}_{1}\dot{x}(t) &=& \mathcal{A}_{1}x(t)+\mathcal{B}_{1}u(t)
+\mathcal{D}_{1}v(t)+\mathcal{D}_{\xi}\xi(t)\\ \notag
&&+\mathcal{F}_{1}g(H\mathcal{K}x,t), \label{sysPa}\\
0 &=& \mathcal{A}_{2}x(t)+\mathcal{B}_{2}u(t)+\mathcal{D}_{2}v(t),
\label{sysPb}\\
\Tilde{y}(t) &= & \mathcal{C}x(t)+\mathcal{G}v(t),    \label{sysPc}\\
	   z(t) &= & \mathcal{K}x(t),    \label{sysPd}
\end{eqnarray*}
Next, the algebraic constraint is rearranged by moving 
$\mathcal{B}_{2}u(t)$ to the left-hand side and then combined with the measured output equation to define the augmented output $y(t) = 
\begin{bmatrix}-\mathcal{B}_{2}u(t) \\ \tilde{y}(t)\end{bmatrix}$.
Consequently, the above system can be rewritten as 
 \allowdisplaybreaks
\begin{subequations} \label{sys1}
     \begin{eqnarray}
         \mathcal{E}_{1}\dot{x}(t) &= & \mathcal{A}_{1}x(t)+\mathcal{B}_{1}u(t)+\mathcal{D}_{1}v(t)+\mathcal{D}_{\xi}\xi(t)\\ \notag
        && + \mathcal{F}_{1}g(H\mathcal{K}x,t), \label{sys1a}\\
	   y(t) &= & \mathcal{C}_{1}x(t)+\mathcal{G}_{1}v(t),    \label{sys1b}\\
	   z(t) &= & \mathcal{K}x(t),    \label{sys1c}
    \end{eqnarray}
\end{subequations}
where  $\mathcal{E}_{1}$, $\mathcal{A}_{1} \in \mathbb{R}^{n_0 \times n}$, $\mathcal{B}_{1} \in \mathbb{R}^{n_0 \times k}$, $\mathcal{D}_{1} \in \mathbb{R}^{n_0 \times q}$, $\mathcal{{D}}_{\xi} \in \mathbb{R}^{n_0 \times s}$, $\mathcal{F}_{1} \in \mathbb{R}^{n_0 \times l}$, $\mathcal{C}_{1} = \begin{bmatrix}
     \mathcal{A}_{2} \\ \mathcal{C}
\end{bmatrix} \in \mathbb{R}^{r_1 \times n}$, $\mathcal{G}_{1} = \begin{bmatrix}
  \mathcal{D}_{2} \\ \mathcal{G}
 \end{bmatrix} \in \mathbb{R}^{r_1 \times q}$ are constant matrices and $y(t) \in \mathbb{R}^{r_1}$ with $r_1=m+r-n_0$. Furthermore, let $\rank
(\mathcal{G}_{1}) = q_1 \leq q$. Consequently, considering the SVD of $\mathcal{G}_{1}$, we obtain two invertible matrices $\mathcal{U}$ and $\mathcal{V}$ of suitable dimensions such that:  
\begin{eqnarray}\label{SVD}
\mathcal{U}\mathcal{G}_{1}\mathcal{V}=\begin{bmatrix}
\mathbb{I}_{q_1} & 0 \\  0     & 0
\end{bmatrix}.  
\end{eqnarray}
Thus, by pre-multiplying \eqref{sys1b} with $\mathcal{U}$ and assuming that $v(t) = \mathcal{V} \begin{bmatrix}
      v_1(t) \\ v_2(t)
  \end{bmatrix}$, we obtain
  \begin{eqnarray*}
      \begin{bmatrix}
      y_1(t) \\ y_2(t)
  \end{bmatrix}= \begin{bmatrix}
      \mathcal{C}_{11} \\ \mathcal{C}_{12}
  \end{bmatrix}x(t) + \begin{bmatrix}
\mathbb{I}_{q_1} & 0 \\  0     & 0
\end{bmatrix}\begin{bmatrix}
      v_1(t) \\ v_2(t)
  \end{bmatrix},
  \end{eqnarray*}
  where $~\mathcal{U}y(t)=\begin{bmatrix}
      y_1(t) \\ y_2(t)
  \end{bmatrix}$, $~\mathcal{U}\mathcal{C}_{1}=\begin{bmatrix}
      \mathcal{C}_{11} \\ \mathcal{C}_{12}
  \end{bmatrix}$ with suitable dimensional partitions. Thus, system \eqref{sys1} can be rewritten as
\allowdisplaybreaks
 
 \begin{eqnarray*}
         \mathcal{E}_{1}\dot{x}(t) &= &  \mathcal{A}_{1} x(t)+\mathcal{B}_{1}u(t)+\mathcal{D}_{11}v_{1}(t)\\ \notag
        && +\mathcal{D}_{12}v_2(t) +\mathcal{D}_{\xi}\xi(t)+ \mathcal{F}_{1}g(H\mathcal{K}x,t),  \label{syssa}\\
	   y_{1}(t) &= & \mathcal{C}_{11}x(t)+v_1(t),  \label{syssb}   \\
        y_{2}(t) &= & \mathcal{C}_{12}x(t), \label{syssc}\\
	   z(t) &= & \mathcal{K}x(t),   
    \end{eqnarray*}
where $ \mathcal{D}_{1}\mathcal{V}=\begin{bmatrix}
      \mathcal{D}_{11} & \mathcal{D}_{12}
  \end{bmatrix}$. Using $v_1(t)= y_1(t)-\mathcal{C}_{11}x(t)$, we obtain
\allowdisplaybreaks
\begin{subequations} \label{sys2}
     \begin{eqnarray}
         \mathcal{E}_{1}\dot{x}(t) &= & \Phi x(t)+\mathcal{B}_{1}u(t)+\mathcal{D}_{11}y_{1}(t) \notag\\
        && +\mathcal{D}_{12}v_2(t) +\mathcal{D}_{\xi}\xi(t)+ \mathcal{F}_{1}g(H\mathcal{K}x,t), \label{sys2a}\\
	   y_{2}(t) &= & \mathcal{C}_{12}x(t),    \label{sys2b}\\
	   z(t) &= & \mathcal{K}x(t),    \label{sys2c}
    \end{eqnarray}
\end{subequations}
where $\Phi= \mathcal{A}_{1}-\mathcal{D}_{11}\mathcal{C}_{11}$.
  
\noindent We now consider the following two cases. \\   
\noindent\textbf{Case~1:} $\rho \begin{bmatrix}\mathcal{C}_{12}^\top & 
\mathcal{K}^\top\end{bmatrix}^\top \neq \rho(\mathcal{C}_{12}) 
+ \rho(\mathcal{K})$. Then some rows of $\mathcal{K}$ lie in 
the row space of $\mathcal{C}_{12}$, and the corresponding 
components of $z(t)$ are directly available from $y_2(t)$ 
without estimation. Consequently, there exist a permutation 
matrix $\mathcal{Q}$ and matrices $\mathcal{S}_{11}$, 
$\mathcal{P}_1$, and $\mathcal{P}_2$ such that 
$\row(\mathcal{S}_{11}) \cap \row(\mathcal{C}_{12}) = \{0\}$, 
$\row(\mathcal{S}_{11}) \subset \row(\mathcal{K})$, and
$$\mathcal{K} = \mathcal{Q}\begin{bmatrix}
\mathcal{S}_{11} \\ 
\mathcal{P}_{1}\mathcal{S}_{11}+\mathcal{P}_{2}\mathcal{C}_{12}
\end{bmatrix},$$
where $\mathcal{S}_{11}$ has full row rank, say $l_1$. Consequently, system 
\eqref{sys2} reduces to
\begin{subequations}{\label{sys3}}
\begin{eqnarray}
\mathcal{E}_{1}\dot{x}(t) &= & \Phi x(t)+\mathcal{B}_{1}u(t)+\mathcal{D}_{11}y_{1}(t) +\mathcal{D}_{12}v_{2}(t)\nonumber\\
&& +\mathcal{D}_{\xi}\xi(t)+ \mathcal{F}_{1}g(H\mathcal{K}x,t),\label{sys3a}\\
y_{2}(t) &= & \mathcal{C}_{12}x(t), \label{sys3b}\\
z(t) &=& \mathcal{Q}\begin{bmatrix}
\mathcal{S}_{11} \\
\mathcal{P}_{1}\mathcal{S}_{11}+\mathcal{P}_{2}\mathcal{C}_{12}
\end{bmatrix}x(t) = \mathcal{Q} \begin{bmatrix}
z_1(t) \\ z_2(t)
\end{bmatrix},\label{sys3c}
\end{eqnarray}
\end{subequations}
where $z_1(t) = \mathcal{S}_{11} x(t)$ denotes the part of $z(t)$ that 
requires estimation, and $z_2(t) = \mathcal{P}_{1} z_1(t) + 
\mathcal{P}_2 y_{2}(t)$ denotes the part that is directly computable from 
$z_1(t)$ and the measured output $y_{2}(t)$. Therefore, it suffices to 
construct a functional filter only for $z_1(t)$, from which the entire 
functional vector $z(t)$ can be recovered via the relation in \eqref{sys3c}. 
This reduces the filter order from $p$ to $l_1$, yielding a 
reduced-order functional filter design.\\
\noindent\textbf{Case~2:}
$\rho \begin{bmatrix}\mathcal{C}_{12}^\top & \mathcal{K}^\top\end{bmatrix}^\top 
= \rho(\mathcal{C}_{12}) + \rho(\mathcal{K})$. Then row spaces of 
$\mathcal{C}_{12}$ and $\mathcal{K}$ are linearly independent, no component of 
$z(t)$ is directly recoverable from $y_2(t)$, and consequently 
$\mathcal{S}_{11} = \mathcal{K}$, $\mathcal{Q} = \mathbb{I}$, and 
$\mathcal{P}_1$, $\mathcal{P}_2$ are empty matrices in \eqref{sys3c}. This case corresponds to a functional filter with $l_1 = p$.
\begin{remark}\label{rem21}
It is worth noting that the state equation \eqref{sys3a} contains the term 
$\mathcal{D}_{11}y_1(t)$, where $y_1(t)$ is a measured output signal. Since $y_1(t)$ is a known measured quantity, 
it can be treated as an additional known input to the system. 
\end{remark}
    
Now, we introduce the following system, which serves as a functional filter for \eqref{sys3}.
\begin{subequations}\label{obs}
\begin{eqnarray}
\dot{w}(t)&= & \mathcal{N}w(t)+\mathcal{T}\mathcal{B}_{1}u(t)+ \mathcal{T}\mathcal{D}_{11}y_1(t)  \\ \notag
&&+ \mathcal{L}y_2(t)+ \mathcal{T}\mathcal{F}_{1}g(H\hat{z},t),  \label{obsa}\\
\hat{z}(t) &=& \mathcal{Q}(\mathcal{R}w(t) + \mathcal{M}y_{2}(t)), \label{obsb} 
 \end{eqnarray}
 \end{subequations}
 where $w(t) \in \mathbb{R}^{l_{1}}$, $\mathcal{R}=\begin{bmatrix}
	\mathbb{I}_{l_{1}} \\  \mathcal{P}_1 \end{bmatrix}$ , 
	$\mathcal{M}=\begin{bmatrix}
\bar{\mathcal{M}}\\ \mathcal{P}_1\bar{\mathcal{M}}+\mathcal{P}_2
\end{bmatrix}$, and $\bar{\mathcal{M}}$ is a matrix of appropriate dimension.
\begin{remark}\label{rem12}
In \eqref{obs}, the matrices $\mathcal{N}, ~\mathcal{T}, ~\mathcal{L}$, and $\mathcal{M}$ denote the design parameters of the filter. The objective is to determine these matrices such that \eqref{obs} serves as a functional $\mathscr{H}_\infty$ filter for system \eqref{sys3}. Notably, the functional vector $z(t)$ does not change while transforming \eqref{sys} into \eqref{sys3}. Hence, any functional filter for \eqref{sys3} is also one for the original system \eqref{sys}.
   
\end{remark}

  \begin{theorem}\label{thm1}
  Under Assumptions \ref{assum:1} and \ref{assum:2}, system \eqref{obs} is a functional $\mathscr{H}_\infty$ filter for system \eqref{sys}, if the following conditions are satisfied:
    \begin{enumerate}
    \item The unknown matrices $\mathcal{N}$, $\mathcal{T}$, $\mathcal{L}$, and $\mathcal{M}$ satisfy the following matrix equations
    \begin{subequations}\label{cond}
		\begin{eqnarray}
		\mathcal{T}\Phi-\mathcal{L}\mathcal{C}_{12}-\mathcal{N}\mathcal{T}\mathcal{E}_1 &=& 0, {\label{conda}}\\
		\mathcal{T}\mathcal{E}_1 + \mathcal{\bar{M}}\mathcal{C}_{12} - \mathcal{S}_{11} &=& 0, {\label{condb}}\\
            \mathcal{T}\mathcal{D}_{12}&=&0.{\label{condc}}
					\end{eqnarray}
				\end{subequations}
       The above matrix equations discard certain terms from the error dynamics, and consequently, the error dynamics can be represented by equation \eqref{et} below.
     \item The vectors $e_1(t) =  \mathcal{T} \mathcal{E}_1x(t)- w(t)$ and $e(t) = z(t)-\hat{z}(t)$ satisfy the system 
     \begin{subequations} \label{et}
         \begin{eqnarray} 
            \dot{e}_1(t) &=&  \mathcal{N}e_1(t) +  \mathcal{T} \mathcal{D}_{\xi}{\xi}(t) +  \mathcal{T}  \mathcal{F}_1 \Delta g, \label{eta}\\
            e(t) &=&  \mathcal{Q} \mathcal{R}e_1(t), \label{etb}
        \end{eqnarray}  
     \end{subequations}
      where  $\Delta g= g(Hz,t)-g(H\hat{z},t)$, such that:
      \begin{enumerate}
 \item[(2a)] when $\xi(t) \equiv 0$, the estimation error $e(t)$ asymptotically approaches to zero;\label{cond:2a}
\item[(2b)] when $\xi(t)\not\equiv 0$, there exists a scalar $\beta\ge 0$ and a prescribed disturbance attenuation level $\gamma>0$ such that the estimation error satisfies 
$\|e\|^2_{2} <  \gamma^2\|\xi\|^2_{2} + \beta$.\label{cond:2b}
 \end{enumerate}
  \end{enumerate}
  \end{theorem}
  \begin{proof}
      Let the transformation error $e_1(t) =  \mathcal{T} \mathcal{E}_1x(t)- w(t)$. Then, the dynamics of $e_1(t)$ are
      \begin{eqnarray*}{}
         \dot{e}_{1}(t) &=& \mathcal{T} \mathcal{E}_1\dot{x}(t)- \dot{w}(t) \notag \\
			&=& \mathcal{N}e_1(t) +\left(\mathcal{T}\Phi-\mathcal{L}\mathcal{C}_{12}-\mathcal{N}\mathcal{T}\mathcal{E}_1\right)x(t) \notag \\
	& &+\mathcal{T}\mathcal{D}_{12}v_{2}(t)+\mathcal{T} \mathcal{D}_{\xi}{\xi}(t)+ \mathcal{T}\mathcal{F}_{1} \times \notag \\
    && \big(g(Hz,t)-g(H\hat{z},t)\big) \notag \\
   &=& \mathcal{N}e_1(t) +\left(\mathcal{T}\Phi-\mathcal{L}\mathcal{C}_{12}-\mathcal{N}\mathcal{T}\mathcal{E}_1\right)x(t) \notag \\
	& & + \mathcal{T}\mathcal{D}_{12}v_{2}(t)+\mathcal{T} \mathcal{D}_{\xi}{\xi}(t)+ \mathcal{T}\mathcal{F}_{1}\Delta g.
\end{eqnarray*}
Applying conditions \eqref{conda} and \eqref{condc}, we obtain
\begin{eqnarray}\label{e1dot}
\dot{e}_1(t) = \mathcal{N}e_1(t) + \mathcal{T}\mathcal{D}_{\xi}\xi(t) 
+ \mathcal{T}\mathcal{F}_1\Delta g.
\end{eqnarray}
In addition, let $e(t)=z(t)-\hat{z}(t)$. Then
\begin{eqnarray*}  
e(t) &=& \mathcal{K}x(t)-\mathcal{Q}(\mathcal{R} w(t)+\mathcal{M}y_{2}(t)) \nonumber\\
&=& \mathcal{Q}\begin{bmatrix}
	    \mathcal{S}_{11} \\
        \mathcal{P}_{1}\mathcal{S}_{11}+\mathcal{P}_{2}\mathcal{C}_{12}
	\end{bmatrix}x(t)-\mathcal{Q}\big(\mathcal{R}w(t) \nonumber \\
    &&+\begin{bmatrix}
			\bar{\mathcal{M}}\\ \mathcal{P}_1\bar{\mathcal{M}}+\mathcal{P}_2
		\end{bmatrix}{y}_2(t)\big) \nonumber \\
&=& \mathcal{Q}\mathcal{R}\left(e_1(t) - (\mathcal{T}\mathcal{E}_{1} + \bar{\mathcal{M}}\mathcal{C}_{12} - \mathcal{S}_{11})x(t)\right).
	\end{eqnarray*}
    Applying condition \eqref{condb}, the above expression reduces to
\begin{eqnarray}\label{ett}
e(t) = \mathcal{Q}\mathcal{R}e_1(t).
\end{eqnarray}
    Clearly, if condition 1) holds, then \eqref{e1dot} and \eqref{ett} infer that the dynamics of the transformation error $e_1(t)$ are governed by \eqref{eta}, and the estimation error $e(t)$ is given by \eqref{etb}. Therefore, to complete the proof, it is sufficient to show that the rank condition \eqref{rank1} in Assumption \ref{assum:2} guarantees the solvability of system \eqref{cond} with respect to the unknown matrices.\\
    We break the rest of the proof into two steps.\\
    \textbf{Step 1}: Here, we establish that the rank condition \eqref{rank1} is equivalent to 
    \begin{eqnarray}\label{rank}
\rho (\Psi_1) = \rho (\Omega_1),
 \end{eqnarray}
 where \\
 $\Omega_1= \begin{bmatrix}
         \mathcal{E}_{1} & \Phi  & \mathcal{D}_{12} \\
          \mathcal{C}_{12} & 0 & 0 \\
          0 & \mathcal{C}_{12}  & 0\\
          0 & \mathcal-{S}_{11} & 0 
          \end{bmatrix}$ and  $\Psi_1 =\begin{bmatrix} \Omega_1 \\
          \begin{bmatrix}
              \mathcal{S}_{11} & 0 & 0
          \end{bmatrix}\end{bmatrix}$.

Since rank is preserved under pre- and post-multiplication by nonsingular 
matrices, we consider the nonsingular block-diagonal matrix 
$\mathcal{\tilde{P}} = \text{blk-diag}(\mathcal{P}, \mathcal{P}, 
\mathbb{I}_{2r+p})$ and use it to pre-multiply $\Omega$. This transforms 
the subblocks involving $\mathcal{E}$, $\mathcal{A}$, $\mathcal{D}$, $\mathcal{\bar{D}}_{\xi}$, and $\mathcal{F}$ via $\mathcal{P}$, yielding the partitioned form:
    \begin{eqnarray*} 
\rank(\Omega)= \rank(\mathcal{\Tilde{P}} \Omega)
= \begin{bmatrix}
\mathcal{E}_{1}  & \mathcal{A}_{1} & 0 & \mathcal{D}_{1} & 0 & \mathcal{D}_{\xi} & \mathcal{F}_{1} \\
0  & \mathcal{E}_{1} & \mathcal{A}_{1}  & 0 & \mathcal{D}_{1} &0 & 0\\
0  & \mathcal{C}_{1} & 0  & \mathcal{G}_{1} & 0 &0 & 0 \\
0  & 0 & \mathcal{C}_{1} & 0  & \mathcal{G}_{1} &0 & 0 \\
0  & 0  & -\mathcal{K} & 0 & 0 &0 & 0   
\end{bmatrix}.
\end{eqnarray*}
Note that the leading block
$\begin{bmatrix}
\mathcal{E}_{1}  & \mathcal{D}_{\xi} & \mathcal{F}_{1} \end{bmatrix}$ has full row rank $n_0$. Therefore, applying Lemma~\ref{lem2} yields
\begin{eqnarray}\label{non1}
\rank(\Omega)=n_0+ \rank(\Omega_0),
\end{eqnarray}
where $\Omega_0= \begin{bmatrix}
 \mathcal{E}_{1} & \mathcal{A}_{1}  & 0 & \mathcal{D}_{1}\\
 \mathcal{C}_{1} & 0  & \mathcal{G}_{1} & 0 \\
 0 & \mathcal{C}_{1} & 0  & \mathcal{G}_{1} \\
 0  & -\mathcal{K} & 0 & 0    
\end{bmatrix}$.\\

Next, we introduce the nonsingular matrices $\Tilde{U} = \text{blk-diag} ( \mathbb{I}_{n_0}, \mathcal{U}, \mathcal{U}, \mathbb{I}_p )$, $\Tilde{V} = \text{blk-diag} ( \mathbb{I}_{2n}, \mathcal{V}, \mathcal{V})$. Using the invariance of rank under nonsingular transformations, we get $
\rank(\Omega_0)= \rank(\Tilde{U} \Omega_0\Tilde{V})$ and then proceed the following operations:
\begin{enumerate}
\item Substitute the decomposition \eqref{SVD}, $ \mathcal{D}_{1}\mathcal{V}=\begin{bmatrix}
      \mathcal{D}_{11} & \mathcal{D}_{12}
  \end{bmatrix}$, $\mathcal{U}\mathcal{C}_{1}=\begin{bmatrix}
     \mathcal{C}_{11} \\ \mathcal{C}_{12}
  \end{bmatrix}$ into $\Tilde{U}\Omega_0\Tilde{V}$;

\item Pre-multiply matrix $\Tilde{U} \Omega_0\Tilde{V}$ with

\begin{center}
    $\begin{bmatrix}
\mathbb{I}_{n_0}  &  & -D_{11} &   \\
 & \mathbb{I}_{r_1} & &   \\
&  & \mathbb{I}_{q_1} &  \\
&  & &   \mathbb{I}_{r_1 - q_1 + p}
\end{bmatrix}$
\end{center}
which is an elementary block row operation that eliminates the 
$\mathcal{D}_{11}$ block from the first block row. After this elimination, the identity block $\mathbb{I}_{q_1}$ appears twice as a leading full rank subblock;
 \item Apply Lemma \ref{lem2} twice for the full rank matrices $\mathbb{I}_{q_1}$ and obtain
\begin{eqnarray*}
\rank(\Omega_0) = \rank(\Tilde{U}\Omega_0\Tilde{V})\notag = 2q_1+\rho \begin{bmatrix}
          \mathcal{E}_{1} & \Phi  & \mathcal{D}_{12} \\
          \mathcal{C}_{12} & 0 & 0 \\
          0 & \mathcal{C}_{12}  & 0\\
          0 & -\mathcal{K} & 0 
          \end{bmatrix};
     \end{eqnarray*}
 \item Substitute $\mathcal{K} = \mathcal{Q}\begin{bmatrix}
	    \mathcal{S}_{11} \\
        \mathcal{P}_{1}\mathcal{S}_{11}+\mathcal{P}_{2}\mathcal{C}_{12}
	\end{bmatrix}$ and using the fact that $\mathcal{Q}$ is a permutation matrix, we have
    \begin{eqnarray*}
\rank(\Omega_0) = 2q_1+\rho \begin{bmatrix}
           \mathcal{E}_{1} & \Phi  & \mathcal{D}_{12} \\
          \mathcal{C}_{12} & 0 & 0 \\
          0 & \mathcal{C}_{12}  & 0\\
          0 & -\mathcal{S}_{11}  & 0\\
          0 & -\mathcal{P}_{1}\mathcal{S}_{11}-\mathcal{P}_{2}\mathcal{C}_{12} & 0
          \end{bmatrix}
     \end{eqnarray*}
     where $\Phi = \mathcal{A}_1 - \mathcal{D}_{11}\mathcal{C}_{11}$ as 
defined in \eqref{sys2};
     \item Apply the elementary row operations corresponding to pre-multiplying the rightmost matrix in the expression above by
     \begin{center}
         $\begin{bmatrix}
            \mathbb{I} & & & & \\
            & \mathbb{I} & & & \\
            & & \mathbb{I} & & \\
            & & & \mathbb{I} & \\
            & & \mathcal{P}_2 &-\mathcal{P}_1 & \mathbb{I}
        \end{bmatrix}$.
     \end{center}
     
    \end{enumerate}
 Therefore, \eqref{non1} becomes
 \begin{eqnarray}\label{r1}
\rank(\Omega) = n_0+ 2q_1+\rank(\Omega_1).
 \end{eqnarray}
By a similar calculation, we obtain 
\begin{eqnarray} \label{r2}
\rank(\Psi) = n_0+ 2q_1 + \rank (\Psi_1).
\end{eqnarray}
Hence, \eqref{r1} and \eqref{r2} imply that the rank condition \eqref{rank1} is equivalent to the rank condition \eqref{rank}. \\
\textbf{Step 2}: In this step, we show \eqref{cond} is solvable if and only if \eqref{rank} holds.\\
Noting that \eqref{conda} is nonlinear in the unknown matrices, we reduce it to a linear one by substituting 
$\mathcal T\mathcal E_{1}=\mathcal S_{11}-\bar{\mathcal M}\mathcal C_{11}$ from \eqref{condb} into \eqref{conda}. 
Consequently, \eqref{conda} reduces to
\begin{eqnarray}\label{ta}
\mathcal{T} \Phi+\mathcal{J}\mathcal{C}_{12}-\mathcal{N} \mathcal{S}_{11} =  0,
\end{eqnarray}	
where $\mathcal{J}  = \mathcal{N}\bar{\mathcal{M}}-\mathcal{L}$. 
Clearly, \eqref{condb}, \eqref{condc} and \eqref{ta} can be expressed in the matrix form
\begin{eqnarray}\label{psitheta}
\begin{bmatrix}
   \mathcal{T} & \bar{\mathcal{M}} & \mathcal{J} & \mathcal{N}
\end{bmatrix}\Omega_1  =  \Theta, 
\end{eqnarray} 
where $ \Theta=
\begin{bmatrix}
\mathcal{S}_{11} & 0& 0
\end{bmatrix}$ and $\Omega_1$ is same as in \eqref{rank}. It follows from Lemma~\ref{lem4} that the solvability of system \eqref{psitheta} is equivalent to the rank condition \eqref{rank}. Furthermore, the corresponding general solution of \eqref{psitheta} is
\begin{eqnarray}\label{gensol}
\begin{bmatrix}
 \mathcal{T} & \mathcal{\bar{M}} & \mathcal{J} & \mathcal{N}
\end{bmatrix}  =  \mathcal{L}_1-\mathcal{Z}\mathcal{L}_2,
\end{eqnarray}
where $\mathcal{L}_1 = \Theta \Omega_{1}^+$, $\mathcal{L}_2 = (\mathbb{I}-\Omega_{1} \Omega_{1}^+)$, and $\mathcal{Z}$ represents an arbitrary matrix of suitable size. Therefore, by taking
\begin{align*}
\mathcal T_1 &= \mathcal L_1\begin{bmatrix} \mathbb{I} & 0 & 0 & 0 \end{bmatrix}^{\!\top}, \qquad
\mathcal T_2 = \mathcal L_2\begin{bmatrix} \mathbb{I} & 0 & 0 & 0 \end{bmatrix}^{\!\top},\\
\bar{\mathcal M}_1 &= \mathcal L_1\begin{bmatrix} 0 & \mathbb{I} & 0 & 0 \end{bmatrix}^{\!\top}, \qquad
\bar{\mathcal M}_2 = \mathcal L_2\begin{bmatrix} 0 & \mathbb{I} & 0 & 0 \end{bmatrix}^{\!\top},\\
\mathcal J_1 &= \mathcal L_1\begin{bmatrix} 0 & 0 & \mathbb{I} & 0 \end{bmatrix}^{\!\top}, \qquad
\mathcal J_2 = \mathcal L_2\begin{bmatrix} 0 & 0 & \mathbb{I} & 0 \end{bmatrix}^{\!\top},\\
\mathcal N_1 &= \mathcal L_1\begin{bmatrix} 0 & 0 & 0 & \mathbb{I} \end{bmatrix}^{\!\top},\; \text{and}\;
\mathcal N_2 = \mathcal L_2\begin{bmatrix} 0 & 0 & 0 & \mathbb{I} \end{bmatrix}^{\!\top}.
\end{align*}
We obtain
\begin{subequations} \label{allsol}
\begin{eqnarray} 
\mathcal{T} &=& \mathcal{T}_1 - \mathcal{Z}\mathcal{T}_2, ~ \bar{\mathcal{M}} = \bar{\mathcal{M}}_1 - \mathcal{Z}\bar{\mathcal{M}}_2,  \label{allsola} \\
\mathcal{J} &=& \mathcal{J}_1 - \mathcal{Z}\mathcal{J}_2, ~ \mathcal{N} = \mathcal{N}_1 - \mathcal{Z}\mathcal{N}_2. \label{allsolb}
\end{eqnarray}
\end{subequations}
Thus, under the fulfillment of the rank condition \eqref{rank}, the error system evolves according to \eqref{et}. In addition, if condition 2) is satisfied, then system \eqref{obs} is a functional $\mathscr{H}_\infty$ filter for \eqref{sys3}, cf. Definition \ref{def2}. This completes the proof.
  \end{proof}
  \begin{remark} Static filter (zero‑order filter): \label{rem13}
If $\rho \begin{bmatrix}
	    \mathcal{C}_{12}^\top & \mathcal{K}^\top	
        \end{bmatrix}^\top =  \rho (\mathcal{C}_{12})$,
then Lemma~\ref{lem4} infers that there exists a constant matrix $\mathcal{P}_{2}$ such that
$\mathcal{K} = \mathcal{P}_{2}\mathcal{C}_{12}$. Accordingly, a zero-order filter for $z(t)$ can be constructed as
$$
\hat z(t) = \mathcal{P}_{2} y_2(t),
 ~~\text{where}~~
 y_2(t) = \mathcal{C}_{12}x(t).
$$
Thus, the resulting estimation error is given by
$$
e(t)
= z(t)-\hat z(t) 
= (\mathcal{K}- \mathcal{P}_{2}\mathcal{C}_{12})x(t) 
= 0, \quad \forall t \geq 0.
$$
Hence, the functional output $z(t)$ can be recovered precisely and instantaneously from the measured output $y_2(t)$. In other words, the rank condition $\rho \begin{bmatrix}
	    \mathcal{C}_{12}^\top & \mathcal{K}^\top	
        \end{bmatrix}^\top =  \rho (\mathcal{C}_{12})$ guarantees that $y_2(t)$ contains all the information needed to reconstruct $z(t)$, so no dynamic state is required.
\end{remark}

Now, the design of the functional $\mathscr{H}_{\infty}$ filter \eqref{obs}, with the parameter matrices given in \eqref{allsol}, reduces to determining the free matrix parameter $\mathcal Z$ such that condition~2) of Theorem~\ref{thm1} is satisfied. To this end, the following result employs Lyapunov stability theory to derive an LMI-based condition for computing $\mathcal{Z}$.\\
In the following theorem, we show that under Assumption \ref{assum:1} and the rank condition \eqref{rank}, if a certain LMI is satisfied, then there exists a $\mathcal H_\infty$ filter of the form \eqref{obs} for system \eqref{sys}.

    \begin{theorem}\label{thm2}
    Consider the nonlinear system \eqref{sys} that satisfies the Assumption~\ref{assum:2}, and the nonlinear function $g$ satisfies Assumption \ref{assum:1}. Then, for a given $\gamma>0$, system \eqref{obs} is a functional $\mathcal{H_{\infty}}$ filter with parameter matrices \eqref{allsol} and error dynamics \eqref{et}, if there exists a scalar $\mu>0$, matrices $\Upsilon=\Upsilon^\top>0$, and $\Upsilon_{1}$ such that the following LMI holds

    \begin{eqnarray}{\label{lmi}}
\Xi = \begin{bmatrix}
{\Xi}_{11}+\bar{R}^\top\bar{R} +\mu \mathbb{I}& {\Xi}_{12}  &  {\Xi}_{13} \\
{\Xi}_{12}^\top & \kappa\mathcal{W}_{22} & 0\\
{\Xi}_{13}^\top & 0 & -\gamma^2\mathbb{I}

     \end{bmatrix} &\leq& 0,
				\end{eqnarray}
				where 
     \begin{subequations}\label{cof}
\begin{eqnarray}
{\Xi}_{11} &=& \mathcal{N}_{1}^\top \Upsilon+\Upsilon \mathcal{N}_{1}-\mathcal{N}_{2}^\top \Upsilon_{1}^\top-\Upsilon_{1}\mathcal{N}_{2} \notag \\
&&+\bar{H}^\top\kappa\mathcal{W}_{11}\bar{H}, {\label{cof1}}\\
{\Xi}_{12} &=& \Upsilon (\mathcal{T}_1 - \mathcal{Z}\mathcal{T}_2)\mathcal{F}_1+\bar{H}^\top\kappa\mathcal{W}_{12},{\label{cof2}}\\
{\Xi}_{13} &=&  \Upsilon(\mathcal{T}_1 - \mathcal{Z}\mathcal{T}_2)\mathcal{D}_{\xi}, \label{cof4} \\
 \bar{H} &=& H\mathcal{Q}\mathcal{R},\; \bar{R}=\mathcal{Q}\mathcal{R}, \; \text{and} \; {\Upsilon_{1}}=\Upsilon \mathcal{Z}. {\label{cof5}}
				\end{eqnarray}
     \end{subequations}
    \end{theorem}
\begin{proof}
    Let LMI \eqref{lmi} be true. Now, we split the proof into the following steps.\\
    \noindent \textbf{Step 1:} Here, we show that, when $\xi(t) \not\equiv 0$, the error dynamics \eqref{et} satisfy the condition (2b) of Theorem \ref{thm1}.\\
               Consider the Lyapunov function as 
\begin{eqnarray}\label{vv}
\mathscr{V}(t) = e_1^\top(t) \Upsilon e_1(t),
    \end{eqnarray}
   where $\Upsilon=\Upsilon^\top>0$. Therefore, 
\begin{eqnarray}\label{vdot}
    \dot{\mathscr{V}}(t) &= & \dot{e}_1^\top(t) \Upsilon e_1(t)+e_1^\top(t) \Upsilon \dot{e}_1(t) \notag\\
 &=& (\mathcal{N}e_1(t) +  \mathcal{T} \mathcal{D}_{\xi}{\xi}(t) +  \mathcal{T}  \mathcal{F}_1 \Delta g)^\top \Upsilon e_1(t) \notag\\
   &&+ e_1^\top(t) \Upsilon (\mathcal{N}e_1(t) +  \mathcal{T} \mathcal{D}_{\xi}{\xi}(t) +  \mathcal{T}  \mathcal{F}_1 \Delta g)        \notag\\
   &=& e_1^\top(t)( \Upsilon\mathcal{N}+\mathcal{N}^\top \Upsilon)e_1(t) \notag \\
   & &+ \xi^\top(t)( \Upsilon\mathcal{T}\mathcal{D}_{\xi})^\top e_1(t) +e_1^\top(t)(\Upsilon\mathcal{T}\mathcal{D}_{\xi})\xi(t) \notag \\
   &&+ \Delta g^\top(\Upsilon\mathcal{T}  \mathcal{F}_1)^\top e_1(t)+e_1^\top(t)(\Upsilon\mathcal{T}  \mathcal{F}_1)\Delta g \notag \\
   &= & \eta^\top(t)
   \begin{bmatrix}
    \mathcal{N}^\top \Upsilon+ \Upsilon\mathcal{N}  &  \Upsilon\mathcal{T}  \mathcal{F}_1 & \Upsilon\mathcal{T}\mathcal{D}_{\xi}\\
(\Upsilon\mathcal{T}  \mathcal{F}_1)^\top & 0 &  0\\
	(\Upsilon\mathcal{T}\mathcal{D}_{\xi})^\top & 0 &  0 \\
	\end{bmatrix}
  \eta(t), \nonumber\\
    \end{eqnarray} where $\eta(t)= \begin{bmatrix}
		  e_1(t) \\
     \Delta g(t)\\
         \xi(t)   
	\end{bmatrix}$. 
   Since the nonlinear function $g$ satisfy the $\delta \mathrm{QC}$ condition \eqref{ineq}, \emph{i.e.}, inequality \eqref{ineq} holds for the for $\delta \mathrm{MM}$ \eqref{imm}, Remark \ref{rem11} allows us to write 
  $$ \begin{bmatrix}
    \zeta_1- \zeta_2 \\
g( \zeta_1, t)-g( \zeta_{2}, t)
\end{bmatrix}^\top \kappa \mathcal{W}\begin{bmatrix}
 \zeta_1- \zeta_2 \\
g( \zeta_1, t)-g( \zeta_{2}, t)
\end{bmatrix} \geq 0 .$$
Taking $\zeta_1=Hz$ and $\zeta_2=H\hat{z}$, we get
\begin{eqnarray*}
   \begin{bmatrix}
    He(t)\\
\Delta g(t)
\end{bmatrix}^\top \begin{bmatrix}
\kappa \mathcal{W}_{11} & \kappa \mathcal{W}_{12} \\
\kappa\mathcal{W}_{12}^\top & \kappa\mathcal{W}_{22}
 \end{bmatrix} \begin{bmatrix}
He(t) \\
\Delta g(t)
\end{bmatrix} \geq 0.
\end{eqnarray*}
   Using \eqref{etb}, inequality \eqref{ineq} can be rewritten as
   \begin{eqnarray*}\label{ineqnew}
  \begin{bmatrix}
    H\mathcal{Q}\mathcal{R}e_1(t)\\
\Delta g(t)
\end{bmatrix}^\top \begin{bmatrix}
\kappa\mathcal{W}_{11} & \kappa \mathcal{W}_{12} \\
\kappa\mathcal{W}_{12}^\top & \kappa\mathcal{W}_{22}
 \end{bmatrix} \begin{bmatrix}
    H\mathcal{Q}\mathcal{R}e_1(t)\\
\Delta g(t)
\end{bmatrix} \geq 0,
\end{eqnarray*}
 which is equivalent to
\begin{eqnarray}\label{ineq2}
\eta^\top(t)
    \begin{bmatrix}
\bar{H}^\top\kappa\mathcal{W}_{11}\bar{H} & \bar{H}^\top\kappa\mathcal{W}_{12}&0 \\
\kappa\mathcal{W}_{12}^\top\bar{H} & \kappa\mathcal{W}_{22}& 0\\
0 & 0& 0
 \end{bmatrix} 
\eta(t) \geq 0,
\end{eqnarray}
where $\bar{H}=H\mathcal{Q}\mathcal{R}$. \\
Combining \eqref{vdot} and \eqref{ineq2}, and using \eqref{allsol}, we obtain 
\begin{eqnarray}\label{vlast}
\dot{\mathscr{V}}(t)  \leq \eta^\top(t)  \begin{bmatrix}
{\Xi}_{11} & {\Xi}_{12}  &  {\Xi}_{13} \\
{\Xi}_{12}^\top & \kappa\mathcal{W}_{22} & 0\\
{\Xi}_{13}^\top & 0 & 0
\end{bmatrix}\eta(t).
\end{eqnarray}
Now, since $e(t) = \mathcal{Q}\mathcal{R}e_1(t)$ from \eqref{etb}, we have
\begin{eqnarray*}
e^\top(t)e(t) &=& e_1^\top(t)\mathcal{R}^\top\mathcal{Q}^\top
\mathcal{Q}\mathcal{R}e_1(t) \\
&=& e_1^\top(t)\bar{R}^\top\bar{R}e_1(t),
\end{eqnarray*}
where $\bar{R} = \mathcal{Q}\mathcal{R}$. Adding 
$e^\top(t)e(t) - \gamma^2\xi^\top(t)\xi(t)$ to both sides 
of \eqref{vlast} and substituting the above expression, 
we obtain
\begin{eqnarray*}
&&\dot{\mathscr{V}}(t) - \gamma^2\xi^\top(t)\xi(t) 
+ e^\top(t)e(t) \\
&\leq& \eta^\top(t)\begin{bmatrix}
{\Xi}_{11} & {\Xi}_{12}  &  {\Xi}_{13} \\
{\Xi}_{12}^\top & \kappa\mathcal{W}_{22} & 0\\
{\Xi}_{13}^\top & 0 & 0
\end{bmatrix}\eta(t) \\
&&+ e_1^\top(t)\bar{R}^\top\bar{R}e_1(t) 
- \gamma^2\xi^\top(t)\xi(t)\\
&=& \eta^\top(t)\begin{bmatrix}
{\Xi}_{11}+\bar{R}^\top\bar{R} & {\Xi}_{12} & {\Xi}_{13} \\
{\Xi}_{12}^\top & \kappa\mathcal{W}_{22} & 0\\
{\Xi}_{13}^\top & 0 & -\gamma^2\mathbb{I} 
\end{bmatrix}\eta(t)\\
&=& \eta^\top(t)\Xi\eta(t) - \mu e_1^\top(t)e_1(t),
\end{eqnarray*}
where $\Xi$ is same as in \eqref{lmi}. Thus, we can write 
\begin{eqnarray*}
    \dot{\mathscr{V}}(t)-\gamma^2\xi^\top(t)\xi(t)+e^\top(t)e(t)\leq - \mu e_1^\top(t)e_1(t)<0. 
\end{eqnarray*}
Consequently, 
\begin{eqnarray*}
    \dot{\mathscr{V}}(t) < \gamma^2\xi^\top(t)\xi(t)-e^\top(t) e(t).
\end{eqnarray*}
\noindent Upon integration of the above inequality from $0$ to $t_f$, it follows that
\begin{eqnarray*}
\int_0^{t_f} \dot{\mathscr{V}}(\sigma) \mathrm{d} \sigma  < \int_0^{t_f} \gamma^2 \xi^\top(\sigma) \xi(\sigma) \mathrm{d} \sigma-\int_0^{t_f} e^\top(\sigma) e(\sigma) \mathrm{d} \sigma,
\end{eqnarray*}
which is equivalent to 
    \begin{eqnarray*}
\mathscr{V}(t_f)-\mathscr{V}(0) < \int_0^{t_f} \gamma^2 \xi^\top(\sigma) \xi(\sigma) \mathrm{d} \sigma-\int_0^{t_f} e^\top({\sigma}) e(\sigma) \mathrm{d} \sigma.
    \end{eqnarray*}
 Since $\mathscr{V}(t_f) \geq 0$, one obtains
 \begin{eqnarray}\label{v1}
    \int_0^{t_f} e^\top(\sigma) e(\sigma) \mathrm{d} \sigma < \gamma^2   \int_0^{t_f} \xi^\top(\sigma) \xi(\sigma) \mathrm{d} \sigma + \beta.
 \end{eqnarray}
 Finally, setting $\beta = \mathscr{V}(0)$ in \eqref{v1} yields $ \|e\|^2_{2} <  \gamma^2\|\xi\|^2_{2} + \beta$. This establishes that condition (2b) of Theorem~\ref{thm1} is fulfilled whenever \eqref{lmi} is satisfied.\\
 \noindent \textbf{Step 2:} This step demonstrates that, when $\xi(t)\equiv 0$, the error dynamics \eqref{et} fulfill the condition (2a) of Theorem \ref{thm1}. In fact, setting $\xi(t)\equiv0$ in \eqref{et} reveals that $e(t) = \mathcal{Q} \mathcal{R}e_1(t)$ and 
          \begin{eqnarray}\label{net}
            \dot{e}_1(t) &=&  \mathcal{N}e_1(t) +   \mathcal{T}  \mathcal{F}_1 \Delta g, 
        \end{eqnarray} 
        and Eq. \eqref{vlast} reduces to
        \begin{eqnarray}\label{vlast1}
\dot{\mathscr{V}}(t)  \leq \begin{bmatrix}
		  e_1(t) \\
     \Delta g(t)   
	\end{bmatrix}^\top  \begin{bmatrix}
{\Xi}_{11} & {\Xi}_{12}  \\
{\Xi}_{12}^\top & \kappa\mathcal{W}_{22} 
\end{bmatrix}\begin{bmatrix}
		  e_1(t) \\
     \Delta g(t)  
	\end{bmatrix}.
\end{eqnarray}
        Thus, $\dot{\mathscr{V}}<0$ if
        \begin{eqnarray}\label{lmi2}
             \Xi_1 =\begin{bmatrix}
{\Xi}_{11}+\bar{R}^\top\bar{R} + \mu \mathbb{I}& {\Xi}_{12}  \\
{\Xi}_{12}^\top & \kappa\mathcal{W}_{22} 
\end{bmatrix} \leq 0,
        \end{eqnarray}
       which is a principal sub-matrix of $\Xi$. Therefore, if \eqref{lmi} holds, the error dynamics \eqref{net} are asymptotically stable, and consequently $e(t)$ approaches $0$ as $t$ approaches $\infty$. This concludes the proof.
\end{proof}
  
\begin{remark}
The matrices $\mathcal{Q}$, $\mathcal{R}$, and consequently 
$\bar{R} = \mathcal{Q}\mathcal{R}$ are pre-computed constant 
matrices obtained from the known system matrices 
$\mathcal{K}$ and $\mathcal{C}_{12}$ via the two-case 
analysis described below system \eqref{sys2}. Hence, 
$\bar{R}^\top\bar{R}$ is a known constant matrix in 
\eqref{lmi} and does not introduce any bilinearity into the 
formulation. Therefore, the condition \eqref{lmi} in 
Theorem~\ref{thm2} remains linear with respect to the 
decision variables.
\end{remark}

\begin{remark}\label{rem14}
From Assumption \ref{assum:1}, positive semi-definite (PSD) multipliers 
$\mathcal{W}$ satisfy the quadratic constraint \eqref{ineq}. However, such 
multipliers become inadmissible under the proposed filter design framework once 
the LMI feasibility condition \eqref{lmi} is imposed. Indeed, if $\mathcal{W}\geq 0$, then for any $\kappa>0$ the scaled matrix $\kappa\mathcal{W}\geq 0$, so every principal submatrix of $\kappa\mathcal{W}$ is PSD; in particular, $\kappa\mathcal{W}_{22}\geq 0$. On the other hand, feasibility of \eqref{lmi} requires $\Xi\leq 0$, so every principal submatrix of $\Xi$ must be negative semi-definite. Since the $(2,2)$ block of $\Xi$ is $\kappa\mathcal{W}_{22}$, this forces $\kappa\mathcal{W}_{22}\leq 0$, contradicting $\kappa\mathcal{W}_{22}\geq 0$ unless $\mathcal{W}_{22}=0$. Hence PSD multipliers cannot ensure feasibility of the proposed LMIs. Similarly, negative semi-definite multipliers violate the $\delta\mathrm{QC}$ \eqref{ineq} for any nonlinearity. Therefore, admissible multiplier matrices in the proposed $\delta\mathrm{QC}$ framework must necessarily be indefinite, with the $(2,2)$ block satisfying $\mathcal{W}_{22}\leq 0$. This block structural condition 
$\mathcal{W}_{22}\leq0$ arises naturally from the LMI feasibility requirement 
$\Xi\leq0$ and is consistent with the commonly used $\delta$MM structures listed 
in Table~\ref{table:IMM}. Furthermore, the scaling parameter $\kappa>0$ acts as an additional tuning parameter that may influence the numerical feasibility of the resulting LMIs without altering the validity of the quadratic constraint. A suitable value of $\kappa$ can be obtained through a simple trial procedure over a positive range of $\kappa$, where the LMI feasibility problem is solved for different values of $\kappa$ using standard MATLAB solvers such as \texttt{feasp} or \texttt{mincx}.
\end{remark}

\begin{remark}
The computational complexity of each filter is assessed in 
terms of floating-point operations (flops), where a $p\times q$ 
matrix addition requires $pq$ flops and a $p\times q$ by 
$q\times r$ matrix multiplication requires $pr(2q-1)$ flops. 
Since the filters under comparison are developed for system structures analogous to \eqref{sys}, the computational complexity analysis is restricted to the filter update equations, ensuring a fair and consistent comparison with the existing approaches \cite{abbaszadeh2012generalized,darouach2011,meng2024observer, sharma2026functional}. 
The resulting complexity expressions for the proposed 
functional filter \eqref{est} alongside the 
reference filters are presented in Table~\ref{table:complexity}.
\end{remark}
\begin{table*}[ht]
\caption{Computational complexity comparison of different filters.}
\label{table:complexity}
\tablefont
\centering
\begin{tabular*}{\textwidth}{@{\extracolsep{\fill}}p{150pt}p{320pt}<{\raggedright}@{}}
\hline\\[-1.5pc]\hline
\centerline{\textbf{Filters}} &
\centerline{\textbf{Total Flops based on filter equations}} \\
\centerline{Full-order Filter \cite{abbaszadeh2012generalized}$^{\ast}$} &
\centerline{$2n^2 + 4mn + 2nr + 2mr + 2ml - 2m - n$} \\
\centerline{Reduced-order Filter \cite{darouach2011}} &
\centerline{$4n^2 + 4kn + 2ln - 2nr - 2kr - 2lr - 2n + r$} \\
\centerline{Reduced-order Filter \cite{meng2024observer}} &
\centerline{$8n^2 + 2kn + 2ln - 4nn_0 - 4nr - 2kr - 2lr + 2n_0r - 2n + r$} \\
\centerline{Functional Filter \cite{sharma2026functional}} &
\centerline{$2p^2+4pk+4pr+2pl-2p$} \\
\centerline{Proposed Functional Filter} &
\centerline{$2l_1^2 + 2l_1k + 2l_1r + 2l_1l + 2pl_1 + 2pk + 2pr - l_1 - p$} \\
\hline\\[-1.5pc]\hline
\multicolumn{2}{l}{$^{\ast}$The output nonlinearity term in \cite{abbaszadeh2012generalized} is set to zero for consistency with \eqref{sys}.}
\end{tabular*}
\end{table*}

   Algorithm~\ref{alg} describes the design procedure of the proposed filter, which is carried out using the MATLAB LMI Toolbox.
\begin{algorithm}
\caption{Algorithmic steps for the design of the filter \eqref{obs} corresponding to the system \eqref{sys}. }\label{alg}
\begin{algorithmic}
\Require The coefficient matrices of system \eqref{sys}
\If{Assumption \ref{assum:1} and Assumption \ref{assum:2} hold,}\\
\quad \textbf{Step $1$:} Transform system \eqref{sys} in the form of system \eqref{sys3}.
\\\quad \textbf{Step $2$:} If $\rho \begin{bmatrix}
	    \mathcal{C}_{12}^\top & \mathcal{K}^\top
	\end{bmatrix}^\top \neq \rho(\mathcal{C}_{12}) + \rho (\mathcal{K})$, compute the matrices $\mathcal{Q}$, $\mathcal{S}_{11}, \mathcal{P}_1$ and $\mathcal{P}_2$ such that  $\mathcal{K} = \mathcal{Q}\begin{bmatrix}
	    \mathcal{S}_{11} \\
        \mathcal{P}_{1}\mathcal{S}_{11}+\mathcal{P}_{2}\mathcal{C}_{12}
	\end{bmatrix}$. Otherwise, set $\mathcal{S}_{11} = \mathcal{K}, \mathcal{Q} = \mathbb{I}$, $\mathcal{P}_1$ and $\mathcal{P}_2$ are empty matrices.\\
    \quad \textbf{Step $3$:}  Find $\mathcal{T}_i$, $\mathcal{\bar{M}}_i$, $\mathcal{J}_i$, and $\mathcal{N}_i$ for $i=1,2$ using the expression given above Eq. \eqref{allsol}.
    \If{LMI \eqref{lmi} is feasible for a given $\gamma$}\\
    \quad \textbf{Step $4$:} Compute matrices $\Upsilon=\Upsilon^\top>0$ and $\Upsilon_{1}$.\\
    \quad \textbf{Step $5$:} Compute $\mathcal{Z} = \Upsilon^{-1}\Upsilon_{1}$. \\
    \quad \textbf{Step $6$:} Compute $\mathcal{T}$, $\mathcal{\bar{M}}$, $\mathcal{J}$, and $\mathcal{N}$ using \eqref{allsola}-\eqref{allsolb}.\\
    \quad \textbf{Step $7$:} Compute $\mathcal{L}  = \mathcal{N}\bar{\mathcal{M}}-\mathcal{J}$.
    \Else{\quad LMI \eqref{lmi} is not feasible.}
\EndIf
\Else{ \quad Filter parameter matrices can not be computed via this approach.}
\EndIf
\end{algorithmic}
\end{algorithm}

\section{Numerical illustration}\label{sec:num_example}
The current section gives two numerical examples to show that the proposed functional $\mathscr{H}_{\infty}$ filter \eqref{obs} works perfectly for the system \eqref{sys}.

\begin{example}\label{ex1}
To assess the effectiveness of the proposed framework, a single-link flexible-joint robotic manipulator is considered. The equations of motion governing this electromechanical system, taken from \cite{liu2023state}, are given by:
\allowdisplaybreaks
\begin{figure}[!t]
\centering
\includegraphics[width = 0.5\textwidth]{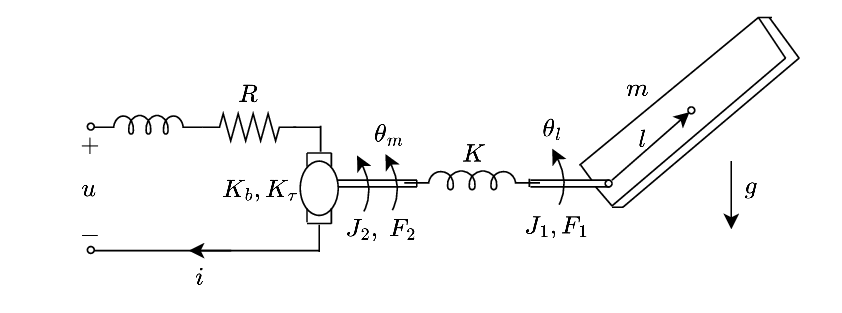}
\caption{A single-link robotic manipulator}
\label{fig:fig1}
\end{figure}
 \begin{eqnarray}\label{link}
J_1 \ddot{\theta}_l+F_1 \dot{\theta}_l+K\left(\theta_l-\frac{\theta_m}{k}\right)+m g l \cos \left(\theta_l\right) &=& 0, \nonumber \\
J_2 \ddot{\theta}_m+F_2 \dot{\theta}_m+\frac{K}{k}\left(\theta_l-\frac{\theta_m}{k}\right) &= &K_\tau i,\nonumber \\
R i+K_b \dot{\theta}_m &=& u,
\end{eqnarray}

where $\theta_m$ and $\theta_l$ are the motor and link angular positions, respectively; $u$ is the applied voltage, and $i$ represents the armature current. Define the state vector as $x = 
\begin{bmatrix}
\theta_l & \dot{\theta}_l & \theta_m & \dot{\theta}_m & i
\end{bmatrix}^\top$, where $\dot{\theta}_m$ and $\dot{\theta}_l$ correspond to the angular velocities of the DC motor and the link, respectively, as illustrated in Fig. \ref{fig:fig1}. To evaluate the robustness and performance of the proposed filter, the spring force is assumed to be influenced by an external disturbance $\xi(t)$; accordingly, system \eqref{link} is directly expressed in the descriptor system form \eqref{sys3} as
\begin{eqnarray}\label{exp}
\begin{bmatrix}
\mathbb{I}_{4} & 0 \\ 0 & 0 
\end{bmatrix}\dot{x}(t)= \begin{bmatrix}
0 & 1 & 0 & 0 & 0 \\
\frac{-K}{J_1} & \frac{-F_1}{J_1} & \frac{K}{J_1k} & 0 & 0 \\ 0 & 0 & 0 & 1 & 0 \\
\frac{-K}{J_2k} & 0 & \frac{K}{J_2k^2} & \frac{-F_2}{J_2} & \frac{K_{\tau}}{J_2} \\
0 & 0 & 0 & -K_b & -R
\end{bmatrix} x(t) \nonumber \\
 +\begin{bmatrix}
0_{4 \times 1} \\ 1
\end{bmatrix}u(t)+ \begin{bmatrix}
    0 \\ 0.2 \\ 0 \\0.5 \\ 0.1
\end{bmatrix} \xi(t)+\begin{bmatrix}
0 \\ \frac{-mgl}{J_1}  \\ 0_{3 \times 1}
\end{bmatrix}\cos x_1(t),~~~~~
  \end{eqnarray}
where $g(H\mathcal{K}x,t) = \cos x_1(t)$ with $H =\begin{bmatrix}
    1 & 0
\end{bmatrix} $. As no unknown input is considered in this example, the matrices $\mathcal{D}_{11}, \mathcal{D}_{12}, \mathcal{C}_{11}$, and vector $v(t)$ do not appear in the model. Let us assume the output and the functional vector are taken as $$y_2(t)=\begin{bmatrix}
    1 & 1 & 0 & 0 & 1 \\
    0 & 0 & 0 & 0 & 1
\end{bmatrix} x(t) ~\text{and}~ z(t)=\begin{bmatrix} 1 & 0_{1 \times 3} & 0\\
0 &0_{1 \times 3}& 1\end{bmatrix}x(t),$$ respectively. The parameters utilized in the simulations are enumerated in the TABLE~\ref{table:1}.
\begin{table}
\caption{Model parameters of the single-link.}
\label{table}
\tablefont
\begin{tabular*}{20.25pc}{@{}p{40pt}p{81pt}<{\raggedright}p{90pt}<{\raggedright}@{}}
\hline\\[-1.5pc]\hline
\centerline{Parameters}& 
\centerline{Description}& 
\centerline{Value} \\
\centerline{$J_1, J_2$} & \centerline{Moment of inertias} & \centerline{$1.625\mathrm{~kg}\mathrm{m}^2$} \\
   \centerline{$F_1, F_2$} &  \centerline{Coefficient of viscous friction} & \centerline{$0.9\mathrm{~Nms}/ \mathrm{rad}$} \\
    \centerline{$K$} & \centerline{Spring constant} & \centerline{$0.5868 \mathrm{Nm} / \mathrm{rad}$} \\
    \centerline{$K_\tau$} & \centerline{Torque constant} & \centerline{ $1 \mathrm{Nm} / \mathrm{A}$} \\
    \centerline{$K_b$} & \centerline{Back emf constant} & \centerline{$1 \mathrm{Nm} / \mathrm{A}$} \\
    \centerline{$R$} & \centerline{Armature resistance} & \centerline{$1\,\Omega$} \\
    \centerline{$l$} & \centerline{Link center of gravity position} & \centerline{$0.5 \mathrm{~m}$} \\
   \centerline{$g$} & \centerline{Acceleration of gravity} & \centerline{$9.8 \mathrm{~N} / \mathrm{kg}$} \\
     \centerline{$m$} & \centerline{Link mass} & \centerline{$4.43 \mathrm{~kg}$} \\
    \centerline{$k$} & \centerline{Gear ratio} & \centerline{$2 $} \\
\hline\\[-1.5pc]\hline
\end{tabular*}
\label{table:1}
\end{table}

Note that the nonlinear function $g$ satisfies the $\delta \mathrm{QC}$ condition described in Assumption \ref{assum:1}, and the corresponding $\delta \mathrm{MM}$ of $g$ is shown in the following
\begin{eqnarray*}
    \kappa\mathcal{W} =\kappa \begin{bmatrix}
        1 &0 \\ 0 &-1
    \end{bmatrix},
\end{eqnarray*}
where the scale factor is set to $\kappa = 1$. 
Since the system \eqref{exp} fulfills Assumption~\ref{assum:2}, Algorithm~\ref{alg} is employed to design the filter.\\
\noindent \textbf{System transformation:}
Since the above system is already expressed in the form \eqref{sys3}, Step~$2$ of Algorithm~\ref{alg} can be directly applied. From Step~$2$, it follows that $\rho \begin{bmatrix}
	    \mathcal{C}_{12}^\top & \mathcal{K}^\top
	\end{bmatrix}^\top = 3\neq 4= \rho(\mathcal{C}_{12}) + \rho (\mathcal{K})$. Therefore, Step $2$ guarantees that some part of the $z(t)$ can be reconstructed using the output $y_2(t)$, and therefore, we obtain $$\mathcal{S}_{11}=\begin{bmatrix}
		1  &  0 & 0 & 0& 0
	\end{bmatrix},~\mathcal{P}_1=0,~
	\mathcal{P}_2=\begin{bmatrix}
	    0 & 1
	\end{bmatrix},$$ $\text{ and }\mathcal{Q}=\mathbb{I}_{2}.$
Therefore, \eqref{sys3c} reduces to
\begin{eqnarray}\label{jj1}
    z(t) =  \mathcal{Q}\begin{bmatrix}
   \mathcal{S}_{11} \\\mathcal{P}_1\mathcal{S}_{11}+\mathcal{P}_2\mathcal{C}_{12}
\end{bmatrix}x(t) = \begin{bmatrix}
      \mathcal{S}_{11}x(t) \\\mathcal{P}_2y_2(t)
\end{bmatrix} = \begin{bmatrix}
    z_1(t) \\ z_2 (t)
\end{bmatrix} .
\end{eqnarray}
It is evident that $z_2(t)$ can be directly obtained from the output $y_2(t)$; therefore, the filter is required only to estimate $z_1(t)$.

    \noindent \textbf{Filter parameter determination:}
    By carrying out the subsequent steps of Algorithm~\ref{alg}, the LMI \eqref{lmi} is solved for $\gamma = 1.003$ and $\mu=0.2$  using the \emph{feasp} function in MATLAB LMI toolbox and obtain 
\begin{eqnarray*}
\mathcal{N} &=& 
    -1
,~\mathcal{T} =\begin{bmatrix}
0 & 0 & 1 & 0& 0
\end{bmatrix}, ~\mathcal{M} =
\begin{bmatrix}
    0 & 0\\
    0 & 1
\end{bmatrix},\notag \\
\mathcal{R} &=& \begin{bmatrix} 
    1 & 0
	\end{bmatrix}^\top, ~\text{and} ~\mathcal{L} = \begin{bmatrix}
1 & -1
\end{bmatrix}.
\end{eqnarray*}

\noindent \textbf{Simulation results:}
A simulation was performed using the input $u(t)= \sin (0.3\pi t)$, the disturbance $\xi(t)= 0.25 \sin(0.4t+0.5 \pi)+0.3$, the initial conditions for the system and the filter as $x_0(t) = \begin{bmatrix}
2.5 & 3 & 0 & 1 & -0.945
\end{bmatrix}^\top$ and $w_0(t) = 2$, respectively. 
All simulations were carried out in MATLAB~R2024b, 
where the system was solved using \texttt{ode15s} 
with absolute and relative tolerances set to $10^{-4}$, 
while the proposed filter was implemented using \texttt{ode45}.Fig.~\ref{fig2} illustrates the estimated and actual time responses of $z(t)$, while Fig.~\ref{fig3} presents the convergence behavior of the error. From Figs. \ref{fig2}–\ref{fig3}, it is evident that the estimate $\hat{z}_2(t)$ perfectly matches the actual state $z_2(t)$ because $z_2(t)$ is directly available from the measured output $y_2(t)$ in \eqref{jj1}. Moreover, the estimated state $\hat{z}_1(t)$ closely follows the actual state $z_1(t)$ even in the presence of external disturbances. These observations indicate that the proposed filter effectively mitigates the impact of external disturbances while maintaining a simple and computationally efficient structure.\\

\begin{figure}[!t]
\centering
\includegraphics[width = 0.5\textwidth]{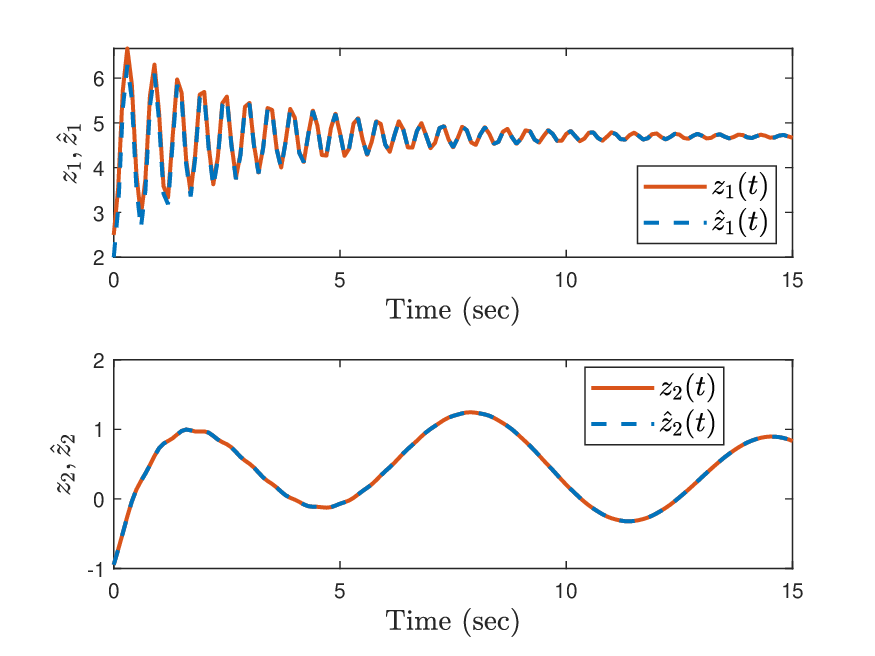}
\caption{Time responses of estimated and actual functional $z(t)$ in the presence of disturbance} 
\label{fig2}
\end{figure}

\begin{figure}[!t]
\centering
\includegraphics[width = 0.5\textwidth]{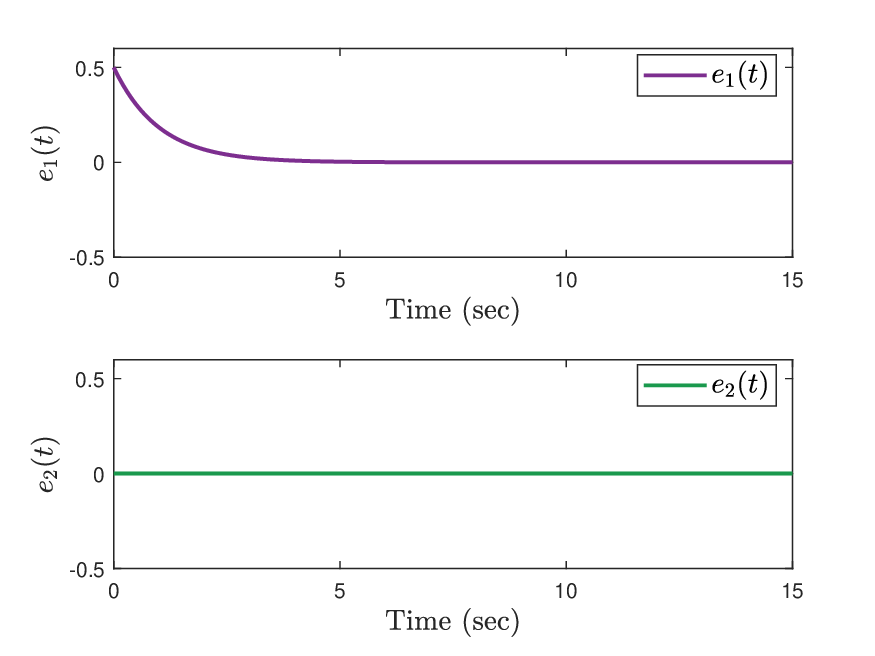}
\caption{Time responses of error $e(t)$} 
\label{fig3}
\end{figure}
\end{example}

\begin{example}
  Consider the descriptor system \eqref{sys} characterized by the following coefficient matrices:
   \begin{eqnarray*}
\mathcal{E} &=&\begin{bmatrix}
1 & 0 & 0 \\
0 & 1 & 0 \\
0 & 0 & 1 \\
 0 & 0 & 0
\end{bmatrix}, \mathcal{A} =\begin{bmatrix}
-9 & 1 & 0 \\
0 & -6 & 1 \\
0 & 0 & -3 \\
0 & 1 & 0
\end{bmatrix},
\mathcal{B} =\begin{bmatrix}
1 \\
0 \\
0 \\
1
\end{bmatrix},  \notag \\ 
\mathcal{D} & =& \begin{bmatrix}
0 \\
1 \\
0 \\
0 
\end{bmatrix},
\mathcal{\bar{D}}_{\xi} =\begin{bmatrix}
2 \\
1 \\
0 \\
0
\end{bmatrix},
\mathcal{F} =\begin{bmatrix}
1 \\
0 \\
0 \\
0
\end{bmatrix}, \mathcal{K} = \begin{bmatrix}
     1 & 0 & 0 \\
     0& 1& 0
\end{bmatrix},\notag \\
\mathcal{C} &=& \begin{bmatrix}
0 & 0 & 1 \end{bmatrix},
 \mathcal{G} =1,  H = \begin{bmatrix}
     1&0
 \end{bmatrix},   \text{and} \; g(t)= e^{-x_1(t)}.
\end{eqnarray*}
We observe that the nonlinear function $g$ fulfills the one-sided Lipschitz condition. The corresponding $\delta\mathrm{MM}$ of $g$ can be obtained as
\begin{eqnarray*}
    \kappa\mathcal{W} =\kappa \begin{bmatrix}
        0 &-0.5 \\ -0.5 &0
    \end{bmatrix},
\end{eqnarray*}
where scaling factor is chosen as $\kappa = 2$. Since, the system also satisfies Assumptions \ref{assum:2}. Hence, we can design a filter for the functional $z(t)$ by Algorithm \ref{alg}.

\noindent \textbf{System transformation:} Applying Step $1$ of Algorithm \ref{alg}, we transform this system into a new coordinate system \eqref{sys3} with the following coefficient matrices. 
\begin{eqnarray*}
\mathcal{E}_{1} &=& \begin{bmatrix}
    -0.92 &   -0.38 & 0   \\ 
    0.38  &  -0.92 & 0 \\
    0 & 0 & 1
\end{bmatrix}, \mathcal{B}_{1} = \begin{bmatrix}
     -0.92 \\ 0.38 \\ 0 
\end{bmatrix}, \notag \\
\Phi &=& \begin{bmatrix}
   8.31 & 1.37 & -0.76 \\
  -3.44 & 5.93 & -1.85 \\
     0 & 0 & -3
\end{bmatrix}, ~\mathcal{F}_{1} = \begin{bmatrix}
   -0.92 \\ 0.38 \\   0
\end{bmatrix}, \notag \\
\mathcal{D}_{\xi} &=& \begin{bmatrix}
   -2.23 \\ -0.16 \\   0
\end{bmatrix},
~\mathcal{D}_{11} = \begin{bmatrix}
   -0.38 \\ -0.92 \\  0
\end{bmatrix}, ~\mathcal{C}_{11} = \begin{bmatrix}
   0 \\ 0 \\ -1
\end{bmatrix}^\top, \notag  \\
\mathcal{C}_{12} &=& \begin{bmatrix}
   0 & 1 & 0
\end{bmatrix}, ~\text{and} ~\mathcal{D}_{12} = \begin{bmatrix}
   \text{empty matrix}
\end{bmatrix}. \end{eqnarray*}
Here,  
 $\rho \begin{bmatrix}
	    \mathcal{C}_{12}^\top & \mathcal{K}^\top
	\end{bmatrix}^\top = 2\neq 3= \rho(\mathcal{C}_{12}) + \rho (\mathcal{K})$. Consequently, Step $2$ guarantees that a portion of the functional vector $z(t)$ can be derived from the output $y_2(t)$. Therefore, we obtain $$\mathcal{S}_{11}=\begin{bmatrix}
		1  &  0 & 0
	\end{bmatrix},~\mathcal{P}_1=0,~
	\mathcal{P}_2=1, \text{ and }\mathcal{Q}=\mathbb{I}_{2}.$$ 
Therefore, \eqref{sys3c} reduces to
\begin{eqnarray}\label{jj}
    z(t) =  \mathcal{Q}\begin{bmatrix}
   \mathcal{S}_{11} \\\mathcal{P}_1\mathcal{S}_{11}+\mathcal{P}_2\mathcal{C}_{12}
\end{bmatrix}x(t) = \begin{bmatrix}
      \mathcal{S}_{11}x(t) \\   y_2(t)
\end{bmatrix} = \begin{bmatrix}
    z_1(t) \\ z_2(t)
\end{bmatrix} .\nonumber\\
\end{eqnarray}
Clearly, we can derive $z_2(t)$ from the output $ y_2(t)$, and filter is required only for $z_1(t)$. 

\noindent \textbf{Filter parameter determination:} By solving the LMI \eqref{lmi} using MATLAB LMI toolbox  for $\mathscr{H}_\infty$ performance level $\gamma = 1.4$ and scalar $\mu =0.3$, and by implementing the subsequent steps of Algorithm~\ref{alg}, the parameter matrices of the proposed filter \eqref{obs} are obtained as follows
\begin{eqnarray*}
\mathcal{N} &=&  -9,~\mathcal{T} =\begin{bmatrix} -0.92 & 0.38 & 0
	\end{bmatrix}, ~\mathcal{M} = \begin{bmatrix}
    0 & 1
	\end{bmatrix}^\top, \notag \\
    \mathcal{R} &=& \begin{bmatrix} 
    1 & 0
	\end{bmatrix}^\top, ~\text{and} ~\mathcal{L} = 1. 
\end{eqnarray*}
\noindent \textbf{Simulation results:}  The estimated and actual values of the functional vector $z(t)$, and their corresponding errors, are depicted in Figs. \ref{fig5}–\ref{fig6}, respectively, for the known input $u(t)= \cos(t)$, unknown input $v(t)=2\sin(t)$, initial condition $x_0(t) = \begin{bmatrix}
1.5 & -1 & 3
\end{bmatrix}^\top$ and $w_0(t) = -5$. Both the system and the proposed filter were solved using MATLAB's \texttt{ode45} solver with absolute and relative tolerances set to $10^{-4}$. In this setup, a bounded random disturbance vector $\xi(t)$ is introduced to model external influences on the system. The disturbance is generated using a uniform random distribution over the interval $[-2,2]$ during the time intervals 
$0 \leq t \leq 5$ and $10 \leq t \leq 15$, and is assumed to be zero at all other time instants, as shown in Fig. \ref{fig4}. The results, presented in Figs. \ref{fig5}–\ref{fig6}, show that the estimated state $\hat{z}_2(t)$ precisely matches with the actual state $z_2(t)$. This is attributed to the fact that $z_2(t)$ is directly included in the output ${y_2(t)}$ from \eqref{jj}. In addition, it is observed that the estimated state $\hat{z}_1(t)$ is significantly less affected by external disturbances than the actual state $z_1(t)$. This indicates that the proposed functional $\mathscr{H}_\infty$ filter design approach exhibits strong disturbance attenuation capability and achieves satisfactory estimation performance.
\begin{figure}[!t]
\centering
\includegraphics[width = 0.5\textwidth]{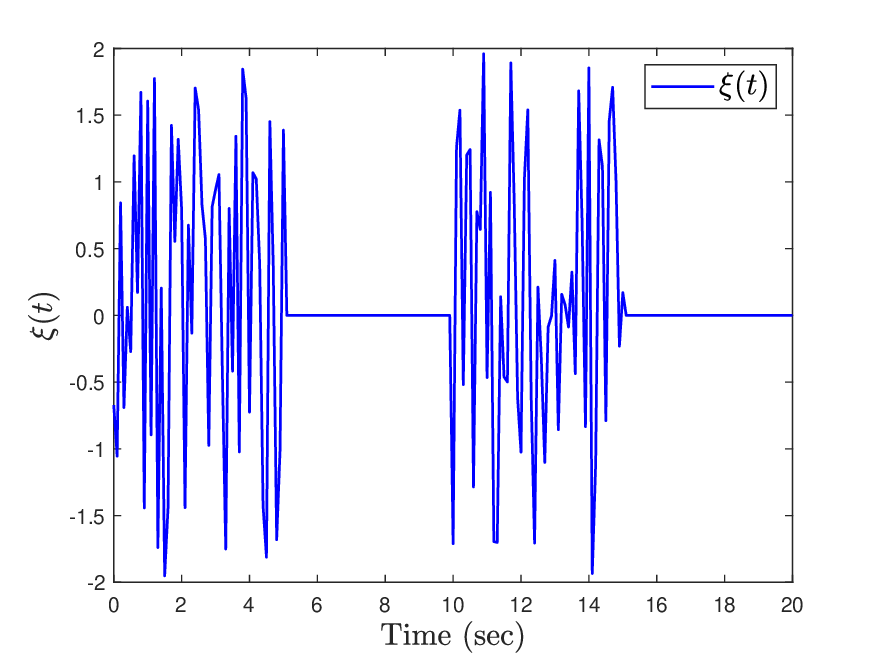}
\caption{Time responses of the bounded uniformly distributed disturbance $\xi(t)$ in $[-2,2]$ for 
the time intervals $0 \leq t \leq 5$ and $10 \leq t \leq 15$, and zero otherwise} 
\label{fig4}
\end{figure}

\begin{figure}[!t]
\centering
\includegraphics[width = 0.5\textwidth]{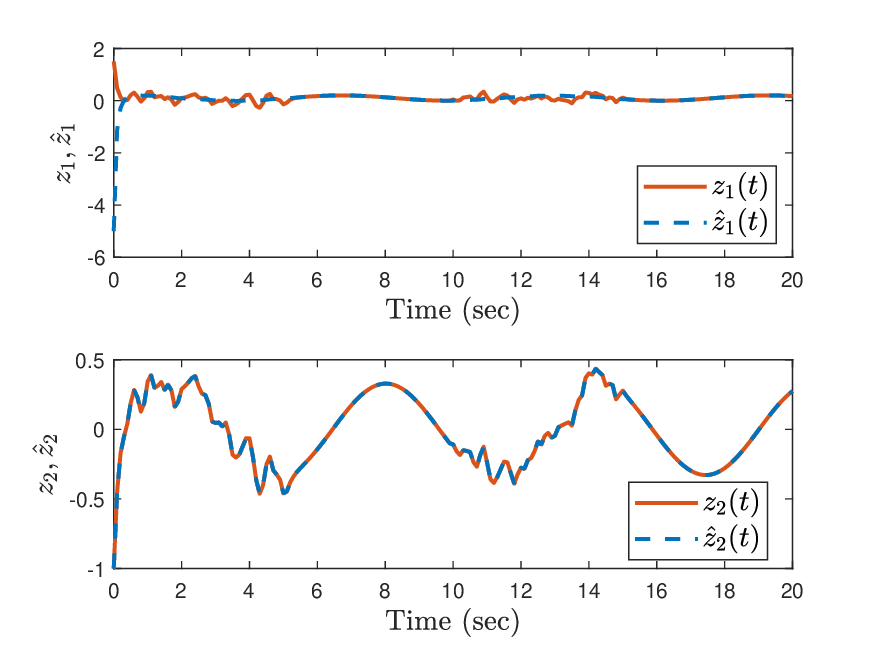}
\caption{Time responses of estimated and actual functional $z(t)$ in the presence of disturbance} 
\label{fig5}
\end{figure}

\begin{figure}[!t]
\centering
\includegraphics[width = 0.5\textwidth]{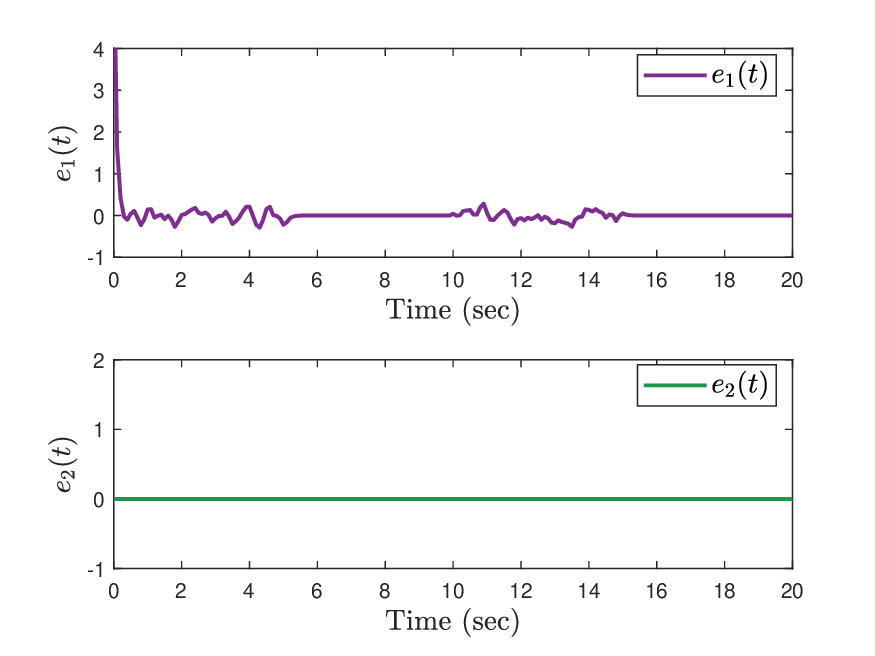}
\caption{Time responses of error $e(t)$} 
\label{fig6}
\end{figure}
\end{example}

\subsection{Comparative Discussion}
Table~\ref{table:comparison} presents a comparative analysis 
of the proposed method with existing filtering approaches in 
terms of filter order, filter form, and nonlinearity 
characterization, and computational complexity. The methods 
in \cite{abbaszadeh2012generalized, darouach2011, meng2024observer} 
rely on the global Lipschitz condition, which may introduce conservatism due to the 
requirement of a uniform incremental bound over the entire 
operating domain. The approach in \cite{sharma2026functional} 
relaxes this assumption through a generalized monotonicity 
condition, whereas the proposed method employs the $\delta$QC 
framework, which encompasses both Lipschitz and monotone 
nonlinearities as special cases, thereby allowing a broader 
class of nonlinear systems to be addressed. Furthermore, the full-order filter in \cite{abbaszadeh2012generalized} is realized in DAE form, which requires consistent initial conditions and imposes 
a higher implementation burden compared to the ODE-based 
realizations adopted by the remaining approaches. As evident 
from Table~\ref{table:comparison}, the proposed method yields 
the lowest flop count per update step, reducing the 
computational cost by approximately $88\%$, $77\%$, $76\%$, 
and $28\%$ compared to~\cite{abbaszadeh2012generalized}, 
\cite{meng2024observer},~\cite{darouach2011}, 
and~\cite{sharma2026functional}, respectively. This reduction 
is attributed to the reduced functional filtering structure, 
which estimates only the desired functional of the state 
rather than the full state vector, making the proposed method 
well-suited for real-time implementation.
\begin{table*}[ht]
\caption{Comparative analysis of existing filtering approaches 
and the proposed method.}
\label{table:comparison}
\tablefont
\centering
\begin{tabular*}{\textwidth}{@{\extracolsep{\fill}}p{120pt}p{70pt}p{80pt}p{85pt}p{70pt}@{}}
\hline\\[-1.5pc]\hline
\centerline{\textbf{Method}} &
\centerline{\textbf{Filter Order}} &
\centerline{\textbf{Filter Form}} &
\centerline{\textbf{Nonlinearity}} &
\centerline{\textbf{Flops (Example-\ref{ex1})}} \\
\centerline{Abbaszadeh \emph{et al.} \cite{abbaszadeh2012generalized}} &
\centerline{Full-order} &
\centerline{DAE} &
\centerline{Lipschitz} &
\centerline{$185$} \\
\centerline{Meng \emph{et al.} \cite{meng2024observer}} &
\centerline{Reduced-order} &
\centerline{ODE} &
\centerline{Lipschitz} &
\centerline{$100$} \\
\centerline{Darouach \emph{et al.} \cite{darouach2011}} &
\centerline{Reduced-order} &
\centerline{ODE} &
\centerline{Lipschitz} &
\centerline{$94$} \\
\centerline{Sharma \emph{et al.}\cite{sharma2026functional}} &
\centerline{Functional Filter} &
\centerline{ODE} &
\centerline{Gen. Monotone} &
\centerline{$32$} \\
\centerline{Proposed Method} &
\centerline{Reduced Functional Filter} &
\centerline{ODE} &
\centerline{$\delta \mathrm{QC}$} &
\centerline{$23$} \\
\hline\\[-1.5pc]\hline
\end{tabular*}
\end{table*}

\section{Conclusions}\label{sec:conclusion}
In this study, we have developed a functional $\mathscr{H}_\infty$ filter design method for nonlinear descriptor systems satisfying the $\delta \mathrm{QC}$ condition in the presence of disturbances. A numerically efficient and easily implementable algorithm is presented for filter design. In cases where redundancy occurs between the functional vector that needs to be estimated and the measured output, the filter order can be chosen lower than the size of the functional vector. We derived sufficient conditions for the existence of the proposed filter in terms of the rank condition \eqref{rank1} and an LMI \eqref{lmi}. Under these conditions, the proposed functional $\mathscr{H}_\infty$ filter renders the error dynamics asymptotically stable, while guaranteeing that the influence of external disturbances on the error remains bounded within a prescribed $\mathcal{L}_{2}$-performance level. However, these conditions are sufficient but not necessary, highlighting the need for further research to derive necessary and sufficient criteria for functional filter existence in nonlinear descriptor systems.

\section*{Appendix}
The estimation problem considered in this paper requires only the
functional $z(t)$ to be uniquely defined, whereas the semistate $x(t)$
of~\eqref{sys} need not be unique. This appendix establishes the rank
condition~\eqref{rank8} underlying the estimation of
$z(t)$ from the available measurements. We show that it is
guaranteed by Assumption~\ref{assum:2} together with the feasibility of
LMI~\eqref{lmi}, and is necessary for the existence of any functional
filter of the form~\eqref{obs}, regardless of the design procedure
employed.

To establish these results, we consider, in addition to the rank
condition~\eqref{rank1} of Assumption~\ref{assum:2}, the following
condition on the coefficient matrices of system~\eqref{sys}:
\begin{equation}\label{rank2}
\rank(\Gamma)=\rank(\Omega),\qquad\forall\,\lambda\in\bar{\mathbb{C}}^{+},
\end{equation}
where $\Omega$ is as in Assumption~\ref{assum:2} and
\begingroup
\setlength{\arraycolsep}{2.5pt}
\begin{equation*}
\Gamma=\begin{bmatrix}
\mathcal{E} & \mathcal{A} & 0 & \mathcal{D} & 0 & \bar{\mathcal{D}}_{\xi} & \mathcal{F}\\
0 & \mathcal{E} & \mathcal{A} & 0 & \mathcal{D} & 0 & 0\\
0 & \mathcal{C} & 0 & \mathcal{G} & 0 & 0 & 0\\
0 & 0 & \mathcal{C} & 0 & \mathcal{G} & 0 & 0\\
0 & \mathcal{K} & \lambda\mathcal{K} & 0 & 0 & 0 & 0
\end{bmatrix}.
\end{equation*}
\endgroup
\begin{remark}\label{rem:rank2}
Condition~\eqref{rank2} is not imposed as an additional assumption on
system~\eqref{sys}. Instead, as shown in Lemma~\ref{lem:detect}, it follows
directly from the feasibility of LMI~\eqref{lmi}.
\end{remark}

The following lemma reduces the rank conditions~\eqref{rank1}
and~\eqref{rank2} on system~\eqref{sys} to equivalent conditions on the
transformed systems~\eqref{sys3} and~\eqref{sys2}, respectively.

\begin{lemma}\label{lem:rank-equiv}
Let $\Omega_1$ and $\Psi_1$ be as in~\eqref{rank}, and define
\begingroup
\setlength{\arraycolsep}{2.5pt}
\begin{gather*}
\Gamma_1=\begin{bmatrix}
\mathcal{E}_{1} & \Phi & \mathcal{D}_{12}\\
\mathcal{C}_{12} & 0 & 0\\
0 & \mathcal{C}_{12} & 0\\
\mathcal{S}_{11} & \lambda\mathcal{S}_{11} & 0
\end{bmatrix},\quad
\Omega_2=\begin{bmatrix}
\mathcal{E}_{1} & \Phi & \mathcal{D}_{12}\\
\mathcal{C}_{12} & 0 & 0\\
0 & \mathcal{C}_{12} & 0\\
0 & -\mathcal{K} & 0
\end{bmatrix},\\[4pt]
\Psi_2=\begin{bmatrix}
\Omega_2\\
\begin{bmatrix}\mathcal{K} & 0 & 0\end{bmatrix}
\end{bmatrix},\quad\text{and}\quad
\Gamma_2=\begin{bmatrix}
\mathcal{E}_{1} & \Phi & \mathcal{D}_{12}\\
\mathcal{C}_{12} & 0 & 0\\
0 & \mathcal{C}_{12} & 0\\
\mathcal{K} & \lambda\mathcal{K} & 0
\end{bmatrix}.
\end{gather*}
\endgroup
Then the following equivalences hold.
\begin{enumerate}
\item Condition~\eqref{rank1} $\iff$ condition~\eqref{rank} $\iff$
\begin{equation}\label{rank5}
\rho(\Psi_2)=\rho(\Omega_2).
\end{equation}
\item Condition~\eqref{rank2} $\iff$
\begin{equation}\label{rank4}
\rho(\Gamma_1)=\rho(\Omega_1),\qquad\forall\,\lambda\in\overline{\mathbb C}^{+},
\end{equation}
$\iff$
\begin{equation}\label{rank6}
\rho(\Gamma_2)=\rho(\Omega_2),\qquad\forall\,\lambda\in\overline{\mathbb C}^{+}.
\end{equation}
\end{enumerate}
\end{lemma}

\begin{proof}
The equivalence between conditions~\eqref{rank1} and~\eqref{rank} has already been established in Step~1 of the proof of Theorem~\ref{thm1}. Therefore, it remains only to prove the remaining equivalences.
\noindent Repeating the transformations used in Step~1 of the proof of
Theorem~\ref{thm1} for the matrix $\Gamma$ yields
\begin{equation}\label{rGamma}
\rho(\Gamma)=n_0+2q_1+\rho(\Gamma_1),
\qquad
\forall\,\lambda\in\overline{\mathbb C}^{+}.
\end{equation}
From~\eqref{r1} and~\eqref{rGamma}, it follows immediately that
\[
\rho(\Gamma)=\rho(\Omega)
\iff
\rho(\Gamma_1)=\rho(\Omega_1).
\]

\noindent Next, before introducing the decomposition of $\mathcal K$ (cf.  Step~1), we apply the same sequence of rank-preserving transformations to obtain
\[
\rho(\Omega)=n_0+2q_1+\rho(\Omega_2),\qquad
\rho(\Psi)=n_0+2q_1+\rho(\Psi_2),
\]
and
\[
\rho(\Gamma)=n_0+2q_1+\rho(\Gamma_2),
\qquad
\forall\,\lambda\in\overline{\mathbb C}^{+}.
\]
It follows from the above identities that
\[
\rho(\Psi)=\rho(\Omega)
\iff
\rho(\Psi_2)=\rho(\Omega_2),
\]
and
\[
\rho(\Gamma)=\rho(\Omega)
\iff
\rho(\Gamma_2)=\rho(\Omega_2).
\]
\end{proof}

The next lemma links the feasibility of LMI~\eqref{lmi} to
condition~\eqref{rank2} through the detectability of the pair
$(\mathcal{N}_1,\mathcal{N}_2)$ defined in~\eqref{allsol}.

\begin{lemma}\label{lem:detect}
Let Assumption~\ref{assum:2} hold, let $\mathcal{W}_{11}\geq0$, and let the
pair $(\mathcal{N}_1,\mathcal{N}_2)$ be as in~\eqref{allsol}. Then:
\begin{enumerate}
\item the feasibility of LMI~\eqref{lmi} implies the detectability of
$(\mathcal{N}_1,\mathcal{N}_2)$;
\item the pair $(\mathcal{N}_1,\mathcal{N}_2)$ is detectable if and only if
condition~\eqref{rank4} holds.
\end{enumerate}
\end{lemma}

\begin{proof}
\emph{Part 1).}
Let LMI~\eqref{lmi} be feasible, i.e., $\Xi\leq0$. Since every principal
submatrix of a negative semidefinite matrix is itself negative
semidefinite, the $(1,1)$ principal block of $\Xi$ satisfies
\begin{equation}\label{detE3}
\Xi_{11}+\bar{R}^{\top}\bar{R}+\mu\mathbb{I}\leq0,
\end{equation}
where $\Xi_{11}$ is given by~\eqref{cof1}. Since $\kappa>0$, $\mathcal{W}_{11}\ge0$, $\bar R^{\top}\bar R\ge0$, and
$\mu I>0$, inequality~\eqref{detE3} implies
\begin{equation}\label{detE5}
\mathcal{N}_1^{\top}\Upsilon+\Upsilon\mathcal{N}_1
-\mathcal{N}_2^{\top}\Upsilon_1^{\top}-\Upsilon_1\mathcal{N}_2<0.
\end{equation}
Substituting $\Upsilon_1=\Upsilon\mathcal{Z}$ into~\eqref{detE5} and using
$\Upsilon=\Upsilon^{\top}>0$ gives
\begin{equation}\label{detE6}
(\mathcal{N}_1-\mathcal{Z}\mathcal{N}_2)^{\top}\Upsilon
+\Upsilon(\mathcal{N}_1-\mathcal{Z}\mathcal{N}_2)<0.
\end{equation}
Thus, $\mathcal N_1-\mathcal Z\mathcal N_2$ is Hurwitz. By
Definition~\ref{def:detect}, this is equivalent to the detectability of
$(\mathcal N_1,\mathcal N_2)$.

\emph{Part 2).}
By Definition~\ref{def:detect}, $(\mathcal{N}_1,\mathcal{N}_2)$ is
detectable if and only if
$\begin{bmatrix}(\mathcal{N}_1-\lambda\mathbb{I})^{\top} &
\mathcal{N}_2^{\top}\end{bmatrix}^{\top}$ has full column rank (FCR) for all
$\lambda\in\overline{\mathbb C}^{+}$. By Lemma~\ref{lem2}, this is
equivalent to the FCR of
\begin{equation}\label{detE8}
\begin{bmatrix}
\mathbb{I} & 0 & 0 & 0\\
0 & \mathbb{I} & 0 & 0\\
0 & 0 & \mathbb{I} & 0\\
\mathcal{T}_1 & \bar{\mathcal{M}}_1 & \mathcal{J}_1 & \mathcal{N}_1-\lambda\mathbb{I}\\
\mathcal{T}_2 & \bar{\mathcal{M}}_2 & \mathcal{J}_2 & \mathcal{N}_2
\end{bmatrix}.
\end{equation}
From the general solution~\eqref{gensol} and the partitions
preceding~\eqref{allsol},
\begin{equation}\label{detE9}
\begin{aligned}
\Theta\Omega_1^{+}&=\begin{bmatrix}\mathcal{T}_1 & \bar{\mathcal{M}}_1 & \mathcal{J}_1 & \mathcal{N}_1\end{bmatrix},\\
\mathbb{I}-\Omega_1\Omega_1^{+}&=\begin{bmatrix}\mathcal{T}_2 & \bar{\mathcal{M}}_2 & \mathcal{J}_2 & \mathcal{N}_2\end{bmatrix}.
\end{aligned}
\end{equation}
Substituting~\eqref{detE9} into~\eqref{detE8}, the FCR of~\eqref{detE8} is
equivalent, for all $\lambda\in\overline{\mathbb C}^{+}$, to the FCR of
\begin{equation}\label{detE10}
\begin{bmatrix}
\begin{bmatrix}
\mathbb{I} & 0 & 0 & 0\\
0 & \mathbb{I} & 0 & 0\\
0 & 0 & \mathbb{I} & 0
\end{bmatrix}\\[3pt]
\Theta\Omega_1^{+}+\begin{bmatrix}0 & 0 & 0 & -\lambda\mathbb{I}\end{bmatrix}\\[3pt]
\mathbb{I}-\Omega_1\Omega_1^{+}
\end{bmatrix}.
\end{equation}
By Lemma~\ref{lem:fcr}, the matrix in~\eqref{detE10} has FCR if and only if,
for all $\lambda\in\overline{\mathbb C}^{+}$,
\begin{equation}\label{detE11}
\rho\!\left(
\begin{bmatrix}
\begin{bmatrix}
\mathbb{I} & 0 & 0 & 0\\
0 & \mathbb{I} & 0 & 0\\
0 & 0 & \mathbb{I} & 0
\end{bmatrix}\\[3pt]
\Theta\Omega_1^{+}+\begin{bmatrix}0 & 0 & 0 & -\lambda\mathbb{I}\end{bmatrix}
\end{bmatrix}\Omega_1\right)=\rho(\Omega_1).
\end{equation}
Expanding~\eqref{detE11} by direct matrix multiplication yields
$\rho(\Gamma_1)=\rho(\Omega_1)$ for all
$\lambda\in\overline{\mathbb C}^{+}$, which is precisely~\eqref{rank4}.
\end{proof}
\begin{remark}\label{rem:chain}
By Lemma~\ref{lem:rank-equiv}, conditions~\eqref{rank4} and~\eqref{rank2}
are equivalent. Hence, the detectability of the pair
$(\mathcal{N}_1,\mathcal{N}_2)$ is also equivalent to
condition~\eqref{rank2}.
\end{remark}

\noindent The next lemma shows that the rank conditions~\eqref{rank5} and \eqref{rank6}
imply condition~\eqref{rank8} on system~\eqref{sys2}.


\begin{lemma}\label{lem:final-rank}
If the rank conditions~\eqref{rank5} and \eqref{rank6} hold, then for all
$\lambda\in\bar{\mathbb{C}}^{+}$,
\begin{equation}\label{rank8}
\rank\begin{bmatrix}\lambda\mathcal{E}_1-\Phi\\ \mathcal{C}_{12}\\
\mathcal{K}\end{bmatrix}
=\rank\begin{bmatrix}\lambda\mathcal{E}_1-\Phi\\ \mathcal{C}_{12}\end{bmatrix}.
\end{equation}
\end{lemma}

\begin{proof}
From~\eqref{rank5} and~\eqref{rank6} we obtain
\begin{equation}\label{frE1}
\rho(\Psi_2)=\rho(\Gamma_2),\qquad\forall\,\lambda\in\bar{\mathbb{C}}^{+}.
\end{equation}
Define
\begingroup
\begin{gather*}
\hat{V}_3=\begin{bmatrix}
\mathbb{I} & \lambda\mathbb{I} & \\
 & -\mathbb{I} & \\
 & & \mathbb{I}
\end{bmatrix},\quad
\hat{U}_3=\begin{bmatrix}
\mathbb{I} & & & \\
 & \mathbb{I} & \lambda\mathbb{I} & \\
 & & -\mathbb{I} & \\
 & & & \mathbb{I}
\end{bmatrix},\\[6pt]
\text{and}\quad
\hat{U}_4=\begin{bmatrix}
\mathbb{I} & & & & \\
 & \mathbb{I} & \lambda\mathbb{I} & & \\
 & & -\mathbb{I} & & \\
 & & & \mathbb{I} & \\
 & & & -\lambda\mathbb{I} & \mathbb{I}
\end{bmatrix}.
\end{gather*}
\endgroup
Since $\hat{U}_3$, $\hat{U}_4$, and $\hat{V}_3$ are nonsingular for all
$\lambda\in\overline{\mathbb{C}}^{+}$, pre- and
post-multiplication by these matrices preserves rank. Hence,
\[
\rho(\hat{U}_4\Psi_2\hat{V}_3)
=
\rho(\hat{U}_3\Gamma_2\hat{V}_3),
\qquad
\forall\,\lambda\in\overline{\mathbb{C}}^{+}.
\]
Equivalently, for every
$\lambda\in\overline{\mathbb C}^{+}$,

\begingroup
\begin{eqnarray*}
\rho\begin{bmatrix}
\mathcal{E}_1 & \lambda\mathcal{E}_1-\Phi & \mathcal{D}_{12}\\
\mathcal{C}_{12} & 0 & 0\\
0 & \mathcal{C}_{12} & 0\\
0 & \mathcal{K} & 0\\
\mathcal{K} & 0 & 0
\end{bmatrix}
=\rho\begin{bmatrix}
\mathcal{E}_1 & \lambda\mathcal{E}_1-\Phi & \mathcal{D}_{12}\\
\mathcal{C}_{12} & 0 & 0\\
0 & \mathcal{C}_{12} & 0\\
\mathcal{K} & 0 & 0
\end{bmatrix}\!.
\end{eqnarray*}
\endgroup
Hence, Lemma~\ref{lem:blockrank} yields
\begin{equation}\label{rank7}
\rank\begin{bmatrix}
\lambda\mathcal{E}_1-\Phi & \mathcal{D}_{12}\\
\mathcal{C}_{12} & 0\\
\mathcal{K} & 0
\end{bmatrix}
=\rank\begin{bmatrix}
\lambda\mathcal{E}_1-\Phi & \mathcal{D}_{12}\\
\mathcal{C}_{12} & 0
\end{bmatrix},
\quad\forall\,\lambda\in\overline{\mathbb{C}}^{+}.
\end{equation}
Applying Lemma~\ref{lem:blockrank} to~\eqref{rank7} yields~\eqref{rank8}, as required. This completes the proof.
\end{proof}

\medskip
\noindent The preceding results show that condition~\eqref{rank8} is
guaranteed by Assumption~\ref{assum:2} together with~\eqref{rank2}. The
following lemma shows that it is also necessary for the existence of a
functional filter: the functional $z(t)$ must be uniquely
determined whenever system~\eqref{sys2} admits a solution.

\begin{lemma}\label{lem:necessary}
Assume that $v_2\equiv0$ and $\xi\equiv0$. If system~\eqref{obs} is a
functional $\mathscr{H}_{\infty}$ filter for system~\eqref{sys2} in the
sense of property~1) of Definition~\ref{def2}, then condition~\eqref{rank8} holds.
\end{lemma}
\begin{proof}
Let system~\eqref{obs} be a functional filter for system~\eqref{sys2} with
$v_2\equiv0$ and $\xi\equiv0$, and assume, to the contrary,
that~\eqref{rank8} does not hold.
Setting $\hat{\mathcal{E}}:=\begin{bmatrix}\mathcal{E}_1\\ 0\end{bmatrix}
\;\text{and}\;
\hat{\mathcal{A}}:=\begin{bmatrix}\Phi\\ \mathcal{C}_{12}\end{bmatrix}$, equations~\eqref{sys2a} and~\eqref{sys2b} (with $v_2\equiv0$,
$\xi\equiv0$) can be rewritten as
\begin{eqnarray}\label{neE1}
\hat{\mathcal{E}}\dot{x}(t)
&=&\hat{\mathcal{A}}x(t)
+\begin{bmatrix}
\mathcal{B}_1u(t)+\mathcal{D}_{11}y_1(t)
+\mathcal{F}_1g(H\mathcal{K}x,t)\\
-y_2(t)
\end{bmatrix} \notag\\
&=:&\hat{\mathcal{A}}x(t)+f(\mathcal{K}x(t),\bar{u}(t),y_2(t)),
\end{eqnarray} where $\bar{u}=\begin{bmatrix}u^{\top} & y_1^{\top}\end{bmatrix}^{\top}$.
Since~\eqref{rank8} does not hold, we have
\begin{equation*}
\ker\begin{bmatrix}\lambda\mathcal{E}_1-\Phi\\ \mathcal{C}_{12}\end{bmatrix}
\not\subseteq\ker\mathcal{K}
\end{equation*}
for some $\lambda_0\in\overline{\mathbb{C}}^{+}$, and hence there exists
$\varphi\neq0$ such that
\begin{equation}\label{neE2}
(\lambda_0\mathcal{E}_1-\Phi)\varphi=0,\quad
\mathcal{C}_{12}\varphi=0,\quad
\mathcal{K}\varphi\neq0,
\end{equation}
or, equivalently, $(\lambda_0\hat{\mathcal{E}}-\hat{\mathcal{A}})\varphi=0$
and $\mathcal{K}\varphi\neq0$.

\emph{Quasi-Kronecker decomposition.} By \cite[Thm.~2.6]{berger2012quasi},
there exist invertible matrices $\mathrm{S}$ and $\mathrm{T}$ such that,
for every $\lambda\in\mathbb{C}$,
\begin{equation}\label{qkf}
\mathrm{S}(\lambda\hat{\mathcal{E}}-\hat{\mathcal{A}})\mathrm{T}
=\text{blk-diag}\big(\lambda\mathcal{E}_P-\mathcal{A}_P,\;
\lambda\mathcal{E}_R-\mathcal{A}_R,\;
\lambda\mathcal{E}_Q-\mathcal{A}_Q\big),
\end{equation}
where the three diagonal blocks satisfy:
\begin{itemize}
\item[(i)] $\lambda\mathcal{E}_{P}-\mathcal{A}_{P}
\in\mathbb{R}^{m_{P}\times n_{P}}$, $m_{P}<n_{P}$, and
$\rank(\lambda\mathcal{E}_{P}-\mathcal{A}_{P})=m_{P}$ for all
$\lambda\in\mathbb{C}$;
\item[(ii)] $\lambda\mathcal{E}_{R}-\mathcal{A}_{R}
\in\mathbb{R}^{n_{R}\times n_{R}}$ is regular, i.e.,
$\det(\lambda\mathcal{E}_{R}-\mathcal{A}_{R})\not\equiv 0$;
\item[(iii)] $\lambda\mathcal{E}_{Q}-\mathcal{A}_{Q}
\in\mathbb{R}^{m_{Q}\times n_{Q}}$, $m_{Q}>n_{Q}$, and
$\rank(\lambda\mathcal{E}_{Q}-\mathcal{A}_{Q})=n_{Q}$ for all
$\lambda\in\mathbb{C}$.
\end{itemize}
 \emph{The free variables reside in Block~$P$.} Writing
$\mathrm{T}^{-1}\varphi
=(\varphi_P^{\top},\varphi_R^{\top},\varphi_Q^{\top})^{\top}$ conformably
with~\eqref{qkf} and premultiplying
$(\lambda_0\hat{\mathcal{E}}-\hat{\mathcal{A}})\varphi=0$ by $\mathrm{S}$ leads to
\begin{gather*}
(\lambda_0\mathcal{E}_P-\mathcal{A}_P)\varphi_P=0,\quad
(\lambda_0\mathcal{E}_R-\mathcal{A}_R)\varphi_R=0,\\
\text{and}\quad
(\lambda_0\mathcal{E}_Q-\mathcal{A}_Q)\varphi_Q=0.
\end{gather*}
The regularity of $\lambda_0\mathcal{E}_R-\mathcal{A}_R$ implies that
$\varphi_R=0$, while \cite[Lem.~3.1]{berger2012quasi} guarantees that
$\lambda_0\mathcal{E}_Q-\mathcal{A}_Q$ has full column rank, and hence
$\varphi_Q=0$. In contrast, for Block~$P$,
$\rank(\lambda_0\mathcal{E}_P-\mathcal{A}_P)=m_P<n_P$, so the null space of
$\lambda_0\mathcal{E}_P-\mathcal{A}_P$ has dimension $n_P-m_P\geq1$, and
nontrivial solutions exist. Hence
\begin{eqnarray*}
\varphi=\mathrm{T}\begin{bmatrix}\varphi_P\\ 0\\ 0\end{bmatrix}, \quad
\varphi_P\in\mathbb{R}^{n_P}, \quad \varphi_P\neq0.
\end{eqnarray*}
If $n_P=0$, then $\varphi_P\in\mathbb{R}^{0}=\{0\}$, contradicting
$\varphi_P\neq0$; so assume $n_P>0$ and define the coordinate change
\begin{equation*}
\begin{pmatrix}x_P\\ x_R\\ x_Q\end{pmatrix}:=\mathrm{T}^{-1}x,
\qquad
\begin{pmatrix}f_P\\ f_R\\ f_Q\end{pmatrix}
:=\mathrm{S}f\big(\mathcal{K}x(t),\bar{u}(t),y_2(t)\big),
\end{equation*}
and partition $\mathcal{K}\mathrm{T}=\begin{bmatrix}\mathcal{K}_P
& \mathcal{K}_R & \mathcal{K}_Q\end{bmatrix}$ conformably, where
\begin{equation}\label{neE4}
z=\mathcal{K}x=\mathcal{K}_Px_P+\mathcal{K}_Rx_R+\mathcal{K}_Qx_Q.
\end{equation}
Now, if $m_P=0$, then Block~$P$ has no rows and $x_P$ may be chosen
arbitrarily. Otherwise, by \cite[Lem.~4.12]{berger2012quasi}, we may assume
without loss of generality that
\begin{equation*}
\lambda\mathcal{E}_P-\mathcal{A}_P
=\lambda\begin{bmatrix}\mathbb{I}_{m_P} & 0\end{bmatrix}
-\begin{bmatrix}\mathrm{N}_P & \mathrm{R}_P\end{bmatrix},
\end{equation*}
with $\mathrm{N}_P\in\mathbb{R}^{m_P\times m_P}$ nilpotent and
$\mathrm{R}_P\in\mathbb{R}^{m_P\times(n_P-m_P)}$. Partitioning
$x_P=(x_P^{1\top},x_P^{2\top})^{\top}$ accordingly, the $P$-block
of~\eqref{neE1} is given by
\begin{equation*}
\dot{x}_P^{1}(t)=\mathrm{N}_Px_P^{1}(t)+\mathrm{R}_Px_P^{2}(t)+f_P(\cdot),
\end{equation*}
whereas no equation governs $x_P^{2}(t)$. Consequently, $x_P^{2}(\cdot)$
may be chosen arbitrarily.

\emph{Two indistinguishable trajectories.} Partition
$\mathcal{K}_P=\begin{bmatrix}\mathcal{K}_P^{1} & \mathcal{K}_P^{2}
\end{bmatrix}$ conformably with $x_P$. From~\eqref{neE2} we have
$\mathcal{K}\varphi=\mathcal{K}_P\varphi_P\neq0$, and hence
$\mathcal{K}_P\neq0$; moreover, $\mathcal{C}_{12}\varphi=0$ yields
\begin{equation}\label{neE5}
\mathcal{C}_P\varphi_P=0,
\end{equation}
where $\mathcal{C}_P$ denotes the block of $\mathcal{C}_{12}\mathrm{T}$
corresponding to $x_P$. Thus, the free component does not affect the
measured output. Since $\mathcal{K}_P\neq0$, we may choose two admissible
trajectories $x^{(1)}$ and $x^{(2)}$ of~\eqref{sys2} that differ only in
$x_P^{2}$ and satisfy
\begin{equation}\label{neE6}
\lim_{t\to\infty}x_P^{2,(1)}(t)=c_1,\quad
\lim_{t\to\infty}x_P^{2,(2)}(t)=c_2,\quad
\mathcal{K}_P^{2}(c_1-c_2)\neq0.
\end{equation}
Such a choice is possible because $x_P^{2}$ is unconstrained, so its
limiting value can be prescribed independently for each of the two
trajectories. By~\eqref{neE5}, both trajectories produce identical outputs,
$y_2^{(1)}\equiv y_2^{(2)}$, whereas~\eqref{neE4} and~\eqref{neE6} imply
\begin{equation}\label{neE7}
\lim_{t\to\infty}\big(z^{(1)}(t)-z^{(2)}(t)\big)
=\mathcal{K}_P^{2}(c_1-c_2)\neq0.
\end{equation}

\emph{Contradiction.} Since the filter~\eqref{obs} is driven by the
identical signals $u$, $y_1$, and $y_2$ in both cases, it produces identical
estimates, $\hat{z}^{(1)}\equiv\hat{z}^{(2)}=:\hat{z}$. Property~1) of
Definition~\ref{def2} requires $\|\hat{z}(t)-z^{(i)}(t)\|\to0$ as
$t\to\infty$ for $i=1,2$. Hence, by the triangle inequality,
\begin{equation*}
\|z^{(1)}(t)-z^{(2)}(t)\|
\le
\|\hat{z}(t)-z^{(1)}(t)\|
+\|\hat{z}(t)-z^{(2)}(t)\|
\longrightarrow0
\end{equation*}
as $t\to\infty$, contradicting~\eqref{neE7}. Hence, condition~\eqref{rank8}
must hold.
\end{proof}

\medskip
\noindent Finally, Lemmas~\ref{lem:rank-equiv}--\ref{lem:necessary} show
that condition~\eqref{rank8} is both guaranteed by the proposed design
conditions and necessary for the existence of a functional filter of the
form~\eqref{obs}.

\bibliography{bibliography}
\end{document}